\documentclass[a4paper,reqno,10pt]{amsart}
\usepackage{amssymb}
\usepackage{amsmath}
\usepackage{amsthm}
\usepackage{color}
\usepackage{esint}
\usepackage{epsfig}
\usepackage{amsfonts}
\usepackage{hyperref}
\usepackage{mathtools}
\usepackage{bm}
\usepackage{tikz}

\usepackage{texdraw}
\usepackage{amsfonts}
\usepackage{caption}
\usepackage{calligra}

\begingroup
\theoremstyle{plain}
\newtheorem{theorem}{Theorem}[section]
\newtheorem{proposition}[theorem]{Proposition}
\newtheorem{lemma}[theorem]{Lemma}
\newtheorem{corollary}[theorem]{Corollary}

\theoremstyle{definition}

\newtheorem{remark}[theorem]{Remark}
\newtheorem{example}[theorem]{Example}
\endgroup

\newcommand{\R}{\mathbb R}
\newcommand{\N}{\mathbb N}

\title{Two extremum problems for Neumann eigenvalues}
\author[Cavallina]{Lorenzo Cavallina}
\address[Lorenzo Cavallina]{Mathematical Institute, Graduate School of Science, Tohoku University, Sendai 980-8578, Japan}
\email{cavallina.lorenzo.e6@tohoku.ac.jp}
\author[Funano]{Kei Funano}
\address[Kei Funano]{Division of Mathematics \& Research Center for Pure and Applied Mathematics, Graduate School of Information Sciences, Tohoku University, Sendai 980-8579, Japan}
\email{kfunano@tohoku.ac.jp}
\author[Henrot]{Antoine Henrot}
\address[Antoine Henrot]{Universit\'e de Lorraine CNRS, IECL, F-54000 Nancy, France}
\email{antoine.henrot@univ-lorraine.fr}
\author[Lemenant]{Antoine Lemenant}
\address[Antoine Lemenant]{Institut Universitaire de France and Universit\'e de Lorraine CNRS, IECL, F-54000 Nancy, France}
\email{antoine.lemenant@univ-lorraine.fr}
\author[Lucardesi]{Ilaria Lucardesi}
\address[Ilaria Lucardesi]{Dipartimento di Matematica, University of Pisa, I-56127 Pisa, Italy }
\email{ilaria.lucardesi@unipi.it}
\author[Sakaguchi]{Shigeru Sakaguchi}
\address[Shigeru Sakaguchi]{Graduate School of Information Sciences, Tohoku University, Sendai 980-8579, Japan.}
\email{sigersak@tohoku.ac.jp}
\date{\today}

\begin{document}

\begin{abstract}
Neumann eigenvalues being non-decreasing with respect to domain inclusion, it makes sense to study the two shape optimization problems
$\min\{\mu_k(\Omega):\Omega \mbox{ convex},\Omega \subset D, \}$ (for a given box $D$) and 
$\max\{\mu_k(\Omega):\Omega \mbox{ convex},\omega \subset \Omega, \}$ (for a given obstacle $\omega$).
In this paper, we study existence of a solution for these two problems in two dimensions and we give some qualitative properties. 
We also introduce the notion of 
{\it self-domains} that are domains solutions of these extremal problems for themselves and give examples of the disk and the square.
A few numerical simulations are also presented.
\end{abstract}

\maketitle

{\small
	
	\bigskip
	\noindent
	
	\bigskip
	\noindent
	\bigskip
	\bigskip
	
\setcounter{section}{0}
\tableofcontents

\section{Introduction}

Let $\Omega\subset \R^2$ be a domain (a connected open set). We consider the two classical eigenvalue problems:

\begin{equation}\label{Dirichlet eigenvalue problem}
\mbox{Dirichlet-Laplacian} \qquad
\begin{cases}
-\Delta u = \lambda u \quad \text{in } \Omega,\\
u=0\quad \text{on }\partial \Omega,
\end{cases}
\end{equation}

\begin{equation}\label{Neumann eigenvalue problem}
\mbox{Neumann-Laplacian} \qquad
\begin{cases}
-\Delta u = \mu u \quad \text{in } \Omega,\\
\partial_n u=0\quad \text{on }\partial \Omega,
\end{cases}
\end{equation}
where $\partial_n$ denotes the directional derivative with respect to $n$, the outward unit normal vector to $\partial\Omega$.
We recall that no smoothness assumption on $\partial \Omega$ is actually needed for the Dirichlet problem \eqref{Dirichlet eigenvalue problem}, stated
in the weak form
 $$
\text{Find } u\in H_0^1(\Omega):\quad \int_\Omega \nabla u\cdot \nabla \varphi = \lambda \int_\Omega u\varphi\quad \text{for all }\varphi\in H_0^1(\Omega).     
$$
On the other hand, some mild regularity (e.g. Lipschitz) is required for the Neumann problem \eqref{Neumann eigenvalue problem} 
to ensure the compactness
embedding from $H^1(\Omega)$ into $L^2(\Omega)$, leading to the variational problem:
$$
\text{Find } u\in H^1(\Omega):\quad \int_\Omega \nabla u\cdot \nabla \varphi = \mu \int_\Omega u\varphi\quad \text{for all }\varphi\in H^1(\Omega).
$$
In this paper, we will be concerned with planar convex domains, therefore this regularity of the boundary holds.
Here, the eigenvalues of problems  \eqref{Dirichlet eigenvalue problem}-\eqref{Neumann eigenvalue problem} will be counted with multiplicity as follows:
\begin{equation*}
\begin{aligned}
    0<\lambda_1(\Omega)<\lambda_2(\Omega)\le \lambda_3(\Omega)\le \dots \lambda_k(\Omega)\le \lambda_{k+1}(\Omega)\le \dots \to \infty, \\
    0=\mu_0(\Omega)<\mu_1(\Omega)\le\mu_2(\Omega)\le \dots \mu_k(\Omega)\le \mu_{k+1}(\Omega)\le \dots \to\infty.
\end{aligned}
\end{equation*}
%When 
%\begin{equation*}
%    \mu_{k-1}(\Omega)<\mu_{k}(\Omega)=\mu_{k+1}(\Omega)=\dots=\mu_{k+\ell-1}(\Omega)<\mu_{k+\ell}(\Omega)
%\end{equation*}
%holds, we say that $\mu_{k}(\Omega),\mu_{k+1}(\Omega),\dots,\mu_{k+\ell-1}(\Omega)$ form a cluster of $\ell$ repeated eigenvalues, $\mu_k(\Omega)$ being the first and $\mu_{k+\ell-1}$ being the last element of the cluster. Moreover, when the above holds true for $\ell=1$, we say that $\mu_k(\Omega)$ is a \emph{simple} eigenvalue. 

{\bf On the monotonicity property of eigenvalues}:
Dirichlet and Neumann eigenvalues share the same homogeneity, but not the same monotonicity.
Indeed, on the one hand we have $\lambda_k(t\Omega)=t^{-2}\lambda_k(\Omega)$ and $\mu_k(t\Omega)=t^{-2} \mu_k(\Omega)$ 
for every $t>0$ and every $k$. On the other hand, as it is well known, Dirichlet eigenvalues are monotonic with respect to set inclusion:
$$\Omega_1 \subset \Omega_2 \Longrightarrow \lambda_k(\Omega_1) \geq \lambda_k(\Omega_2).$$
This is due to the embedding between Sobolev spaces $H^1_0(\Omega_1) \hookrightarrow H^1_0(\Omega_2)$
(together with the fact that the Rayleigh quotient is unchanged by extending functions by zero outside $\Omega_1$) .
Now this monotonicity property is false for Neumann eigenvalues as shown by the following elementary example:
taking a thin rectangle $\Omega_1$ close (from the inside) to the diagonal of a square $\Omega_2$, it is immediate to check that 
$\mu_1(\Omega_1) < \mu_1(\Omega_2)$. This example is represented in Fig. \ref{fig-diagonal}.
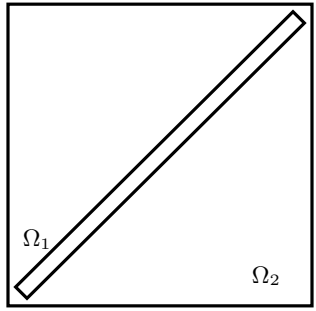
\begin{figure}[h]
\centering
\begin{tikzpicture}
%\draw[blue, very thick] (0,0) rectangle (4,4);
%\draw[red, very thick] (0.25,0.1)--(3.9,3.75)--(3.75,3.9)--(0.1,0.25) -- cycle;
\draw[very thick] (0,0) rectangle (4,4);
\draw[very thick] (0.25,0.1)--(3.9,3.75)--(3.75,3.9)--(0.1,0.25) -- cycle;
%\draw[blue, very thick] (0.05,0.02) rectangle (0.98,0.95);
%[tdplot_main_coords, scale=2.5]
%\draw (0,0,0)--(1,0,0)--(1,1.,0)--(0,1,0)--(0,0,0);
%%\draw (0.05,0.02,0)--(0.95,0.98,0)--(0.98,0.95.,0)--(0.02,0.05,0)--(0.05,0.02,0);
\draw (3.4,0.4,0) node {$\Omega_2$};%{\textcolor{blue}{$\Omega_2$}};
\draw (0.5,,0.3) node {$\Omega_1$}; %{\textcolor{red}{$\Omega_1$}};
%\node(a) at (7,2) {$\mu_1(\Omega_1) < \mu_1(\Omega_2)$};
\end{tikzpicture}
\caption{Example of $\Omega_1 \subset \Omega_2$ with $\mu_1(\Omega_1)< \mu_1(\Omega_2)$.} \label{fig-diagonal}
\end{figure}

Therefore, it makes sense to consider the two following problems for any integer $k\geq 1$.
The {\it interior problem}
$$
(INTP)_k\left\lbrace
\begin{array}{l}
\mbox{Let $D$ be an open convex domain (non-empty), we look for a convex domain}\\
\mbox{ $\Omega_k^*$ \textit{contained into} $D$ solution of}\\
\mu_k(\Omega_k^*)=\min \{\mu_k(\Omega), \Omega \; \hbox{convex}, \Omega \subset D\}
 \end{array}  \right.
$$
and the {\it exterior problem}:
$$
(EXTP)_k\left\lbrace
\begin{array}{l}
\mbox{Let $\omega$ be a, open convex domain (non-empty) we look for a convex domain }\\
\mbox{$D_k^*$ \textit{containing} $\omega$ solution of}\\
\mu_k(D_k^*)=\max \{\mu_k(D), D \; \hbox{convex}, \omega \subset D\}.
 \end{array}  \right.
$$
Thus, for any $k\in \N, k\geq 1$ and any bounded convex domain $D$ we can introduce the following quantity:
\begin{equation}\label{I_k def}
    I_k(D):= \inf\left\{ \mu_k(\Omega) \;:\; 
    \Omega \subset D,\; \Omega \text{ is a convex domain} \right\},
\end{equation}
and for any bounded convex domain $\omega$ we also introduce
%\begin{equation}\label{hatI_k def}
$$
   \hat{I}_k(\omega):= \sup\left\{ \mu_k(D) \;:\; 
    \omega \subset D,\; D \text{ is a convex domain} \right\}.
$$
%\end{equation}
We will see below, see Theorem \ref{theoexisext}, that the above supremum is actually a maximum while it is not necessarily true
for the infimum in \eqref{I_k def}.  In any cases when the minimum or the maximum is achieved, we will denote respectively
by $\Omega_k^*$ and $D_k^*$ the minimizer (for Problem $(INTP)_k$) and the maximizer (for Problem $(EXTP)_k$).

The aim of this paper is to study these two shape optimization problems. In Section \ref{secexistence} we discuss the question of existence.
As already mentioned, we prove that we always have existence for the exterior problem, while for the interior one we prove that we
never have existence for $k=1$ and we give a practical criterion ensuring existence for $k\geq 2$ and give several examples.
Then in Section \ref{secqualitative}, we give some qualitative properties of the optimal domains.
In particular we will be interested in those convex domains that are themselves the solution for some $k$, i.e. they verify
either $I_k(D)=\mu_k(D)$ or  $ \hat{I}_k(\omega)=\mu_k(\omega)$.  Such convex domains will be referred to as 
$k$-\emph{interior self-domains} or $k$-\emph{exterior self-domains}. 
It is not easy to prove that a given domain is a self-domain while it is much more easy to prove
that it is not. In Section \ref{secsquaredisk} we will consider the particular cases of the square and the disk.
We will prove that, for these two examples, we always have existence of an optimal domain (for the interior problem and $k\geq 2$) and moreover we
will give values of the index $k$ for which they can or they cannot be self-domains.

In some sense, one can say that, concerning the minimization problem for the $k$th Neumann eigenvalue, $k$ self-domains are ``better" competitors than any of their convex subdomains,  and thus, they satisfy some kind of ``inner monotonicity property". 
One can also quantify the lack of said ``inner monotonicity property" by the value of the following shape functional:
$$
J_k(D):= \frac{I_k(D)}{\mu_k(D)},    
$$
defined for all bounded convex domains $D\subset \R^2$.
By definition, $0\le J_k(D)\le 1$, with $J_k(D)=1$ holding if and only if $D$ is a $k$-interior self-domain. 
Note that we can also compute the quantity 
$$
\hat{J}_k(\omega):= \frac{\mu_k(\omega)}{\hat{I}_k(\omega)} .
$$
We will study in Section \ref{secJk} the shape minimization problem of finding those convex domains which exhibit the greatest lack of ``inner monotonicity property":
$$
M_k:=\inf \left\{ J_k(D) \;:\; D\text{ is a convex domain of }\R^2 \right\}.
$$
It is clear that it is equivalent to minimize the functional $\hat{J}_k$. We suspect the non-existence of a minimizer, namely
that minimizing sequences converge to a segment and in that case, we can describe the precise behavior of such a minimizing sequence.
In any case, we will give bounds for the infimum  $M_k$ of $J_k$.

At last, in Section \ref{secnum} we give some simple numerical examples in the case of the square or the disk that illustrate 
some properties of the optimal domains.

Let us mention that, while completing the writing of this paper, Pedro Freitas sent us his preprint \cite{Fre-Ken} where the authors are
interested in the same question. In particular, they study (with our notations) the functional $J_k$ and obtain bounds similar to
the ones we present in Section \ref{secJk}.
\section{Existence of optimal domains}\label{secexistence}
In several places in this paper we will use the following result that is an adaptation of a Lemma we can
find in Buser's book, \cite[Section 8.2.1]{Bus}.  The original proof is for Riemann surfaces. 
For the benefit of the reader, we rewrite the proof at the end of the paper (Appendix A), by adapting the original one in our setting.
Let us precise that by $j$-partition of $\Omega$ we mean a collection of sets $\Omega_1,\ldots,\Omega_j$ such that
$$\bar{\Omega}=\bar{\Omega}_1\cup \bar{\Omega}_2\ldots \cup \bar{\Omega}_j$$
 and 
$$\Omega_{j_1} \cap \Omega_{j_2}=\emptyset \quad\mbox{ for any pair $j_1\not= j_2$.}$$

\begin{lemma}[Generalized Buser Lemma]\label{BuserLemma} For any bounded domain $\Omega$ that is decomposed into a $j$-partition 
$\Omega_1,\ldots,\Omega_j$ and for any decomposition of the integer $k$ as the sum of $j$ positive integers: $k=k_1+\ldots + k_j$, we have 
\begin{eqnarray}\label{busineq}
\mu_k(\Omega)\geq \min_i \mu_{k_i}(\Omega_i).
\end{eqnarray}
Moreover, if the inequality \eqref{busineq} is an equality, then $\mu_{k_i}(\Omega_i)=\mu_{k}(\Omega)$ for all $i$ and there exists an eigenfunction $\varphi \in H^1(\Omega)$ associated with $\mu_k(\Omega)$ whose restriction to each $\Omega_i$ is also a Neumann eigenfunction for $\mu_{k_i}(\Omega_i)$. 
\end{lemma}

The following immediate corollary is also useful.

\begin{corollary}[Buser Bound]\label{BuserLemma2} For any bounded domain $\Omega$ that is decomposed into a $k$-partition 
$\Omega_1,\ldots,\Omega_k$ in such a way that each $\Omega_i$ is convex, then we have 
 \begin{eqnarray}
\mu_k(\Omega)\geq \frac{\pi^2}{\max_i \mathrm{diam}^2(\Omega_i)} \nonumber
\end{eqnarray}
where $\mathrm{diam}(D_i)$ denotes the diameter of $D_i$.
\end{corollary}

\begin{proof} We apply Lemma \ref{BuserLemma} with  $k_i=1$ for all $i$, yielding
$$\mu_k(\Omega)\geq \min_{1\leq  i\leq k}(\mu_{1}(\Omega_i)).$$
Now we use the Payne-Weinberger inequality, see  \cite{PW}, which provides a lower bound for $\mu_1$ on convex sets, in terms of the diameter.
More precisely, it says that
$$
\mu_1(\Omega_i) > \frac{\pi^2}{\textrm{diam}^2(\Omega_i)}.
$$
This allows us to conclude that
$$\mu_k(\Omega)\geq   \frac{\pi^2}{\max_{1\leq i \leq k} \mathrm{diam}^2(\Omega_i)}.$$
\end{proof}

\subsection{The interior problem}
We start with a non-existence result that is actually inspired by the counterexample we gave at the beginning.
\begin{theorem}\label{theoI1}
The infimum
$$I_1(D)=\inf\left\{ \mu_1(\Omega) \;:\; 
    \Omega \subset D,\; \Omega \text{ is a convex domain} \right\}$$
 is not attained. Moreover, it is given by 
 $$I_1(D)=\frac{\pi^2}{\mathrm{diam}^2(D)}$$
 where $\mathrm{diam}(D)$ denotes the diameter of $D$.
\end{theorem}
\begin{proof}
First of all, by Payne-Weinberger inequality, see \cite{PW}, for any convex subdomain $\Omega$ of $D$ we have
$$\mu_1(\Omega)>\frac{\pi^2}{\mathrm{diam}^2(\Omega)} \geq \frac{\pi^2}{\mathrm{diam}^2(D)}.$$
For the converse inequality let $[AB]$ a segment realizing the diameter of the closure $\bar{D}$. By the convexity of $D$, for any $\varepsilon>0$
it is possible to construct a thin rectangle of length $\mathrm{diam}(D)-\varepsilon$ contained into $D$. Since its first eigenvalue
is $\pi^2/(\mathrm{diam}(D)-\varepsilon)^2$ the result follows.
\end{proof}
Now, in the general case, $k\geq 2$, we give a useful criterion to prove existence.
\begin{theorem}\label{existencelemma}
    Let $D$ be a planar bounded convex domain. There exists a minimizer for the interior problem \eqref{I_k def} if and only if
    there exists a subdomain $\Omega \subset D$ such that 
\begin{equation}\label{necandsuff}
\mu_k(\Omega)\leq \frac{k^2\pi^2}{\mathrm{diam}^2(D)}
\end{equation}
where $\mathrm{diam}(D)$ denotes the diameter of $D$.
\end{theorem}
\begin{remark}
    Note that for a sequence of thin rectangles $R_\varepsilon$ approaching the diameter of $D$, we have indeed
    $\mu_k(R_\varepsilon) \longrightarrow \frac{k^2\pi^2}{\mathrm{diam}^2(D)}$, the convergence being from above.
\end{remark}
\begin{proof}[Proof of Theorem \ref{existencelemma}]
    According to the previous remark, if there exists a minimizer $\Omega_k^*$ necessarily
    $$\mu_k(\Omega_k^*) \leq \liminf_{\varepsilon\to 0} \mu_k(R_\varepsilon) = \frac{k^2\pi^2}{\mathrm{diam}^2(D)}.$$
    Conversely, let us assume that there exists a subdomain $\Omega$ such that 
    $\mu_k(\Omega)\leq k^2\pi^2/\mathrm{diam}^2(D).$ Let us consider a minimizing sequence $\Omega_n$
    for the minimization problem \eqref{I_k def}. Our aim is to prove that $\Omega_n$ converges to
    a bounded (convex) open set $\Omega^*$
    both for Hausdorff convergence and convergence of characteristic functions. This will imply that
    $\mu_k(\Omega_n) \longrightarrow \mu_k(\Omega^*)$, see \cite[Section 3.7]{HenPie} and therefore
    $\Omega^*$ is a minimizer.
 
Now, assume that the minimizing sequence of convex sets does not converge to a convex open set. Necessarily,
it has to ``collapse" to a segment $\Sigma$. Let us denote by $L$ the length of $\Sigma$: $L\leq \mathrm{diam}(D)$.
In other words, up to a subsequence, $\Omega_n$ will be contained 
in a rectangle $R_n$ of base $\Sigma_n$ of length $L+1/n$ and height $w_n \longrightarrow 0$.
We cut the rectangle $R_n$ in $k$ equal pieces $R_{j,n}$, each cut being done parallel to the shortest side
of length $w_n$,  and denote $\Omega_{j,n}=\Omega_n \cap R_{j,n}$.
Each $\Omega_{j,n}$ being convex, we have by the Payne-Weinberger inequality
$$\mu_1(\Omega_{j,n}) \geq \frac{\pi^2}{\mathrm{diam}^2(\Omega_{j,n})}\geq \frac{\pi^2}{\mathrm{diam}^2(R_{j,n})}=
\frac{\pi^2}{\frac{(L+1/n)^2}{k^2}+w_n^2}.$$
Therefore, using Lemma \ref{BuserLemma}, we finally get
$$\liminf_{n\to \infty} \mu_k(\Omega_n) \geq \lim_{n\to \infty} \frac{\pi^2}{\frac{(L+1/n)^2}{k^2}+w_n^2}= \frac{k^2\pi^2}{L^2}
\geq \frac{k^2\pi^2}{\mathrm{diam}^2(D)}.$$
This shows that $L=\mathrm{diam}(D)$ and $\Omega$ itself must be a minimizer such that $\mu_k(\Omega) = k^2\pi^2/\mathrm{diam}^2(D).$
\end{proof}

\begin{example} A family of convex shapes for which we have existence of a solution to $I_2$ is that of constant width bodies. Let $D$ be a convex body of constant width 1. In particular $\mathrm{diam}(D)=1$. 
In order to check condition \eqref{necandsuff}, we consider as $\Omega\subset D$ an inscribed disk with radius $\rho$, being $\rho$ the inradius of $D$. Using the homogeneity of $\mu_2$ and the exact value of $\mu_2$ at disks, we obtain
$$
\mu_2(\Omega)=\frac{\mu_2(B_1) }{\rho^2}= \frac{(j_{1,1}')^2}{\rho^2}.
$$
The right-hand side is below $4\pi^2$ (we recall that here $\mathrm{diam}(D)=1$) provided that
$$
\rho > \frac{j_{1,1}'}{2\pi}.
$$
This is true since $j_{1,1}' /(2\pi) \sim 0.293$ and since the minimal inradius of a constant width body is that of the Reuleaux triangle, which here (taking the width to be 1) is equal to $1-1/\sqrt{3}\sim 0.4226$.
\end{example}

We now use the criterion \eqref{necandsuff} to prove that, for every domain $D$ there exists a minimizer for $k$ large enough
(and we give a quantitative value for this ``large enough").
\begin{corollary}\label{corexis} 
Given a planar bounded convex domain $D$, there exists $k_0=k_0(D)\in \mathbb N$ such that the minimization problem 
$(INTP)_k$ has a solution for every $k\geq k_0$. 

Moreover we have
\begin{equation}\label{k0D}
 k_0(D)\leq \frac{8\, \mathrm{diam}(D)}{\pi w^\perp},  
 \end{equation} 
where $w^\perp$ denotes the width of $D$ in the direction orthogonal to that of the diameter.
\end{corollary}
\begin{proof}
Let $D$ be a planar convex domain. Without loss of generality, we may assume that the diameter is horizontal. Let $Q_1$ and $Q_2$ be two boundary points, aligned horizontally, satisfying $\overline{Q_1 Q_2}=\mathrm{diam}(D)$. Let $w^\perp$ denote the vertical thickness (that is, the width of $D$ in the direction orthogonal to that of the diameter). Therefore, $D$ is contained into a horizontal strip of thickness $w^\perp$, that without loss of generality is $\mathbb R \times [0, w^\perp]$. Let $Q_3$ and $Q_4$ be two boundary points lying on the supporting lines $y=0$ and $y=w^\perp$. The quadrilateral $\Omega:= Q_1Q_3 Q_2Q_4 $ is contained into $D$ and satisfies
$$
\mathrm{diam}(\Omega)=\mathrm{diam}(D), \quad |\Omega| = \frac{\mathrm{diam}(D)w^\perp}{2}.
$$
Since quadrilaterals are tiling (or plane-covering) domains,  P\'olya's inequality holds true (we recall that it is one of the
most famous conjectures in spectral geometry for general domains), see \cite{Poltil}  and we obtain
$$
\mu_k(\Omega) \leq \frac{4\pi k}{|\Omega|} =\frac{8\pi k}{\mathrm{diam}(D)w^\perp}.
$$
If this upper bound is below $k^2\pi^2/\mathrm{diam}^2(D)$, in view of Proposition \ref{existencelemma}, we have existence of a minimizer for $I_k(D)$. This is true for 
$$
k\geq  \frac{8 \mathrm{diam}(D)}{\pi w^\perp}.
$$
%On the other hand, for the box $D$ we can also use a weaker version of P\'olya inequality due to P. Kr\"oger, see \cite[Corollary 2]{Kr}, namely
%\begin{align*}
%\mu_k(D)\leq \frac{8\pi k}{|D|}
%    \end{align*}
%One more time, if the above quantity is less than $k^2\pi^2/\mathrm{diam}^2(D)$ we conclude to existence: this gives the second 
%possible value for $k_0(D)$.
%Note that, the thinner is the domain, the larger is the threshold $k_0$. 
This concludes the proof.
\end{proof}

\begin{example}\label{exasquare}
For the square, for the disk and for the equilateral triangle, we have $k_0(D)=2$ meaning that the interior problem always has a solution
(for $k\geq 1$).
Indeed, formula \eqref{k0D} provides $k_0\leq 3$ for the square and the equilateral triangle while the inequality \eqref{necandsuff} is directly
verified for $k=2$. For the disk, the fact that the inscribed square has the same diameter allows to conclude directly using the inequalities
for the square.
\end{example}

\begin{remark} 
Excepted for the special case of $k=1$, we have not yet found an example of a domain $D$ for which we have no existence of a minimizer for some index $k\geq 2$.
According to the characterization \eqref{necandsuff}, we should find a domain $D$ for which, for any subdomain $\Omega$, 
we have $\mu_k(\Omega)>k^2\pi^2/\mathrm{diam}^2(D)$.
\end{remark}
\subsection{The exterior problem}
The exterior problem $(EXTP)_k$ shares some common features with the interior problem, but it has also important differences. The first one is
\begin{theorem}\label{theoexisext}
For any convex domain $\omega$, and any $k\geq 1$ the maximization problem $(EXTP)_k$ has a solution.
\end{theorem}
\begin{proof}
Since the inclusion constraint prevents a maximizing sequence to collapse to a segment, the only point
that remains to prove is that this maximizing sequence has a diameter uniformly bounded
(then we use the Blaschke selection theorem  and the continuity of Neumann eigenvalues for the Hausdorff
convergence of convex sets).

Now, by the upper bound proved in \cite{Kr}, \cite{HeMi} for any planar convex domain:
$$\mu_k(\Omega)\leq \frac{(2j_{0,1} + (k-1)\pi)^2}{\textrm{diam}^2(\Omega)}$$
(here $j_{0,1}$ is the first zero of the Bessel function $J_0$)
this shows clearly that the diameter of a maximizing sequence cannot go to $+\infty$.
\end{proof}

\begin{remark}
\begin{itemize}
\item Even for $k=1$, the exterior problem has a solution: this is an important difference with the interior problem.
\item Actually, since a maximizer  $D_k^*$ is a better (or equal) competitor than $\omega$,
we have proved the following bound for the diameter of $D_k^*$:
$$\textrm{diam}(D_k^*) \leq \frac{(2j_{0,1} + (k-1)\pi)}{\sqrt{\mu_k(\omega)}}.$$
\item We can also define the notion of {\it $k$-exterior self-domain} : it is a domain $\omega$ that is itself the solution of the
exterior problem for some $k$. We refer to Section \ref{secsquaredisk} where we prove that the disk and the square are $1$-exterior self domains 
(i.e.  for $\mu_1$).
\end{itemize}
\end{remark}

\section{Qualitative properties of optimal domains}\label{secqualitative}
\subsection{Touching points}
In what follows, when no ambiguity may arise, we will denote a generic minimizer of an interior problem 
as $\Omega^*$ and a generic maximizer of an exterior problem as $D^*$, omitting the
    subscript $k$. Among the immediate properties that these optimal shapes must satisfy, let us mention:
\begin{itemize}
\item $\Omega^*$ must touch the boundary of $D$ in at least two points. Indeed, otherwise we can certainly
translate and then expand the domain $\Omega^*$: by $-2$-homogeneity, this operation strictly decreases the eigenvalue;
\item $D^*$ must touch the boundary of $\omega$ in at least two points. Indeed, otherwise we can certainly
translate and then shrink the domain $D^*$: by $-2$-homogeneity, this operation strictly increases the eigenvalue.
\end{itemize}
In many situations we can say more about the number of ``touching points",  by using the following {\it stretching lemma}
showing the domain monotonicity for Neumann eigenvalues under one dimensional stretching. It
can be found for example in \cite[Proposition 8.1]{LS} or \cite[Lemma 6.6]{Siu2}, the proof being straightforward using
the variational characterization of eigenvalues and a change of variable. For the convenience of the reader, we have provided a proof in the appendix (see Lemma \ref{strechL}).
\begin{lemma}[Stretching]\label{lemstretch}
 Let $\Omega$ be a Lipschitz domain in the plane and write $\Omega_t =\{(x, ty) : (x, y) \in \Omega\}$ for $t \geq  1$, so that $\Omega_t$ 
 is a vertically stretched copy of $\Omega$. Then $\mu_k(\Omega) \geq \mu_k(\Omega_t)$ for each $t \geq 1$.
 \end{lemma}
 This result allows us to deduce that, for the interior problem associated to the square, minimizers must touch the 4 boundary sides of the 
 box, possibly at corners.  It is enough to argue by contradiction: any convex subset of the square not touching the four
 sides can be translated and stretched horizontally or vertically, keeping the constraints (of convexity and inclusion) satisfied,
 but decreasing the eigenvalue (thanks to the stretching lemma). 
%For example, let $D$ be the square for the interior problem. We see that if we would have only two touching points on opposite sides
%of the square, we can stretch or translate then stretch in the orthogonal direction: this shows that for a square,  if an optimal domain
%for the interior problem does exist then we can find  one that  touches the four sides (including the case where the domain touches two opposite vertex).

For the exterior problem, the convexity constraint will also play an important role while determining the number of
touching points. For example, if $\omega$ is the square and if we imagine only two touching points, those
can only be two opposite vertexes of the square. One more time, the stretching lemma (considering now $t<1$), shows that we can modify 
an optimal domain $D^*$ so that it touches the square in (at least) one other point: a third vertex.

\subsection{Multiplicity}\label{secmulti}
For the interior problem, we can prove
\begin{theorem}\label{theoBus1}
Let $\Omega^*$ be a minimizer for a given box $D$ and an index $k$, then $\mu_{k}(\Omega^*) < \mu_{k+1}(\Omega^*)$.
\end{theorem}
\begin{proof}
Assume, for a contradiction, that $\mu_{k}(\Omega^*) = \mu_{k+1}(\Omega^*)$. 
Let us now cut a very small part of $\Omega^*$
(for example a small triangle near a vertex if $\Omega^*$ is (partly) a polygon or a small cap near a strictly convex part. Let us denote by $\Omega_2=T$ this small part and by $\Omega_1$ its complement into $\Omega^*$.
By Lemma \ref{BuserLemma} applied to this decomposition and $k+1$, we have
$$\mu_{k}(\Omega^*) = \mu_{k+1}(\Omega^*) \geq \min\{\mu_k(\Omega_1),\mu_1(\Omega_2)\}.$$
But since we have chosen $\Omega_2$ to be very small (i.e. with a very small diameter), the classical Payne-Weinberger
inequality $\mu_1(\Omega_2) > \pi^2/\textrm{diam}^2(\Omega_2)$ shows that $\mu_1(\Omega_2)$ is very
large, so the minimum in the previous inequality is $\mu_k(\Omega_1)$. Moreover, for $T$ very small, we must have $\mu_k(\Omega_1)\not = \mu_1(T)$. This observation rules out the equality case in the statement of Lemma \ref{BuserLemma}, which finally yields 
$$\mu_{k}(\Omega^*) > \mu_k(\Omega_1),$$
 contradicting the minimality of $\Omega^*$.
\end{proof}
A consequence of this theorem is that we can now easily prove that the disk or the square
are not (interior) self-domains when $k$ is such that $\mu_k=\mu_{k+1}$. We will come back to these examples in Section 
\ref{secsquaredisk}.

\medskip
Now, for the exterior problem, we have a similar property, but we need to add some assumptions on the optimal domain.
\begin{theorem}\label{doubmax}
Let $D^*$ be a maximizer for the exterior problem for a given obstacle $\omega$ and an integer $k$. Assume that
\begin{itemize}
\item either the boundary of $D^*$ contains a strictly convex part
\item or $D^*$ is a polygon with at least one side having a length $\ell$ satisfying $\ell < 2j_{0,1}/\sqrt{\mu_k(D^*)}$
\end{itemize}
then $\mu_{k-1}(D^*) < \mu_k(D^*)$.
\end{theorem}
\begin{proof}
As in the proof of the previous theorem, we start by assuming $\mu_{k-1}(D^*) = \mu_k(D^*)$ for a contradiction.
Let us first assume that the boundary of $D^*$ contains a strictly convex part (it is not important whether
this part is common with the boundary of $\omega$ or not). By strictly convex, here, we mean that we are
able to add to $D^*$ a very small part preserving convexity. This is not possible for example for a polygon
where the convexity constraint will lead us to involve two or three consecutive vertexes.
\begin{center}
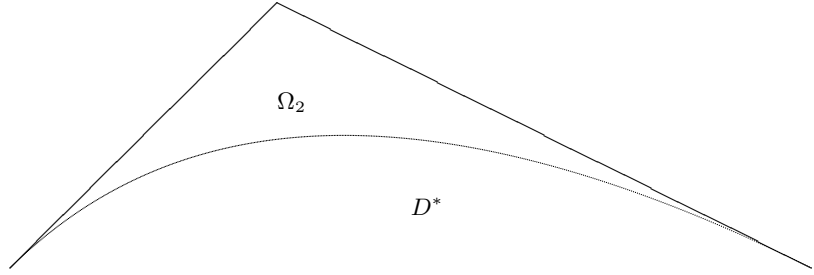
\begin{figure}[h]
%\centering
\begin{picture}(280,120)(0,0)
\qbezier(0,0)(100,100)(300,0)
\put (0,0){\line(1,1){100}}
\put (100,100){\line(2,-1){200}}
\put (100,60){$\Omega_2$}
\put (150,20){$D^*$}
%\put(30,0){Figure: }
\end{picture}
\caption{A small part $\Omega_2$ added on a strictly convex part.}\label{fig-cusps}
\end{figure}
\end{center}

Then, by applying the Generalised Buser Lemma \ref{BuserLemma} to $\Omega_1=D^*$, $\Omega_2$ the small part that we have added (see Fig. \ref{fig-cusps}), and $D=\Omega_1\cup \Omega_2$, we have
$$
\mu_k(D) \geq \min(\mu_{k-1}(D^*),\mu_1(\Omega_2)).
$$
Here, let us remark that $\Omega_2$ has two cusps and therefore, it is not a Lipschitz domain. Nevertheless, we can define
$\mu_1(\Omega_2)$ as the infimum of the usual Rayleigh quotient (among functions in $H^1(\Omega_2)$ orthogonal to constants) and one can check that Buser Lemma also applies with this definition (see the proof of the Lemma in Appendix A).
Moreover, we can prove a Poincar\'e inequality for $\mu_1(\Omega_2)$ showing that $\mu_1(\Omega_2)$ is large when $\Omega_2$
is small (i.e. has small diameter), see Appendix C for more details (note that we cannot use here Payne-Weinberger lower 
bound involving the diameter
since $\Omega_2$ is not convex). Therefore, we deduce 
$$\mu_k(D) > \min(\mu_{k-1}(D^*),\mu_1(\Omega_2))= \mu_{k-1}(D^*)=\mu_{k}(D^*)$$
being a contradiction with the maximality of $D^*$.

Now, if the boundary of $D^*$ has nowhere a strictly convex part, it should be a polygon.
We can  do the same construction as before, by adding to the side of length $\ell$  a small isosceles triangle: the one whose first eigenvalue will converge to $4j_{0,1}^2\ell^2$ when the height goes to zero (see \cite{HeMi}).
If we have  $4j_{0,1}^2\ell^2 > \mu_k(D^*)$ as assumed, we can conclude exactly in the same way. This allows us to get the result for polygons
with small enough sides.
\end{proof}

\begin{corollary}\label{coromu1mu2}
Let $D^*$ be a maximizer for the exterior problem for a given obstacle $\omega$ and $k=2$. 
Then $\mu_{1}(D^*) < \mu_2(D^*)$.
\end{corollary}

\begin{proof}
Let us assume, for a contradiction, that $\mu_{1}(D^*) = \mu_2(D^*)$.
If the boundary of $D^*$ contains a strictly convex part, this is immediately ruled out by Theorem \ref{doubmax}. Therefore, it remains
to consider the case where $D^*$ is a polygon and see if we can apply the second assumption of the theorem in that case.
Let $\ell$ be the length of any side of this polygon: by definition we have $\ell \leq \mathrm{diam}(D^*)$. Moreover, Cheng's inequality,
see \cite{Cheng, HeMi, Kr}, ensures that $\mu_1(D^*)<(2j_{0,1})^2/\mathrm{diam}^2(D^*)$ , therefore if we had 
$\mu_1(D^*)=\mu_2(D^*)$, then it would follow that
$$\ell \leq \mathrm{diam}(D^*) < \frac{2j_{0,1}}{\sqrt{\mu_1(D^*)}}= \frac{2j_{0,1}}{\sqrt{\mu_2(D^*)}}$$
allowing to apply Theorem \ref{doubmax}: a contradiction.
\end{proof}

\section{The square and the disk}\label{secsquaredisk}
In this section, we will apply the previous results (for small values of $k$) when the box or the obstacle are the unit disk or the unit square.
First of all, we recall that we have proved that, in these two cases, we have existence of a minimizer for the interior problem for any $k\geq 2$,
thanks to Corollary \ref{corexis}, see Example \ref{exasquare}. 

\medskip
An interesting question is to know whether the disk or the square could
be or not be {\it self-domains} for both problems. Tables \ref{table1} and \ref{table2} sum up what can be said in view of our previous result,
notably Theorems \ref{theoBus1} and \ref{doubmax}. The first table is for the interior problem, the second for the exterior one,
YES or NO mean that
the disk or the square are or are not self-domains for that value of $k$, {\it probably} that they should be but we cannot prove it.

\begin{center}
\begin{table}[h]
\begin{tabular}{|c|c|c|c|c|} \hline
Domain & $k=1$  & $k=2$  & $k=3$  & $k=4$  \\ \hline
Disk & no existence & probably & NO & probably  \\ \hline
Square & no existence& probably & probably & NO \\ \hline
\end{tabular}
\caption{The interior problem: are disk and square self-domains?}\label{table1}
\end{table}
\begin{table}[h]
\begin{tabular}{|c|c|c|c|c|} \hline
Domain & $k=1$  & $k=2$ & $k=3$ & $k=4$ \\ \hline
Disk & YES & NO & probably & NO \\ \hline
Square & YES & NO & probably & probably \\ \hline
\end{tabular}
\caption{The exterior problem: are disk and square self-domains?}\label{table2}
\end{table}
\end{center}

Explanations: for the disk, we have $\mu_1=\mu_2<\mu_3=\mu_4$, therefore it cannot be optimal for the interior problem for $k=3$
according to Theorem \ref{theoBus1}. Moreover, since the disk is strictly convex, Theorem \ref{doubmax} applies and the disk cannot
be optimal for the exterior problem for $k=2$ and $k=4$. For the square, we have $\mu_1=\mu_2<\mu_3<\mu_4=\mu_5$
therefore it cannot be optimal for the interior problem for $k=4$
according to Theorem \ref{theoBus1}. Moreover, by Corollary \ref{coromu1mu2} the square cannot
be optimal for the exterior problem for $k=2$.
It remains to consider the case $k=1$:
\begin{proposition}\label{propsquare}
The disk and the square are self-domains for the exterior problem and $k=1$.
\end{proposition}
\begin{proof}
Let us start with the the unit disk $\mathbb{U}$: let $\Omega$ be any convex domain strictly containing the unit disk.  On the one hand, the areas satisfy $|\mathbb{U}| < |\Omega|$. On the other hand, the Szeg\H{o}-Weinberger inequality, see \cite{Wein}, 
provides the inequality $|\Omega|\mu_1(\Omega) \leq |\mathbb{U}|\mu_1(\mathbb{U})$. Therefore
$$|\mathbb{U}|\mu_1(\Omega) < |\Omega|\mu_1(\Omega) \leq |\mathbb{U}|\mu_1(\mathbb{U}),$$
proving $\mu_1(\Omega) < \mu_1(\mathbb{U})$ and the optimality of the disk.

Now, let us look at the square $Q=[0,1]\times [0,1]$: let $\Omega$ be any convex domain strictly containing the unit square and let us denote by
$N,W,S,E$ points on the boundary of $\Omega$ that are respectively at the North, the West, the South and the East (for example $N$ is defined
as a point on the boundary and on the horizontal  supporting line above $Q$...), see Fig. \ref{fig-NSWE}. Two of these points might coincide. 
\begin{figure}[h]
\centering
\includegraphics[scale=0.35]{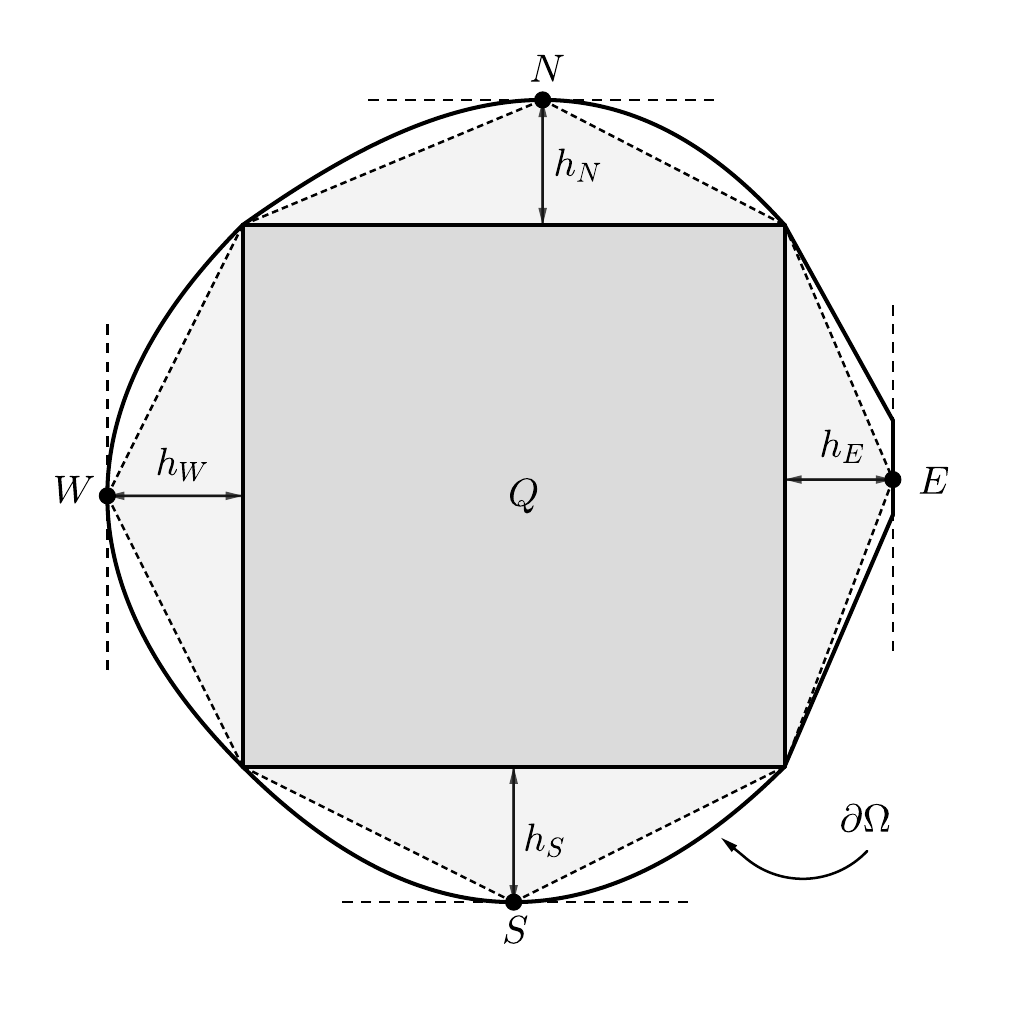}
\caption{Proof of Proposition \ref{propsquare}.}\label{fig-NSWE}
\end{figure}
Let us denote respectively by
$h_N,h_W,h_S,h_E$ the distance between these points and the corresponding side of the square: for example $h_N$ is the difference
between the ordinate of $N$ and $1$. Let us denote by $w_\Omega$ the minimal width of the domain $\Omega$. By definition, we have
$$w_\Omega \leq 1+h_N+h_S,\quad w_\Omega \leq 1+h_W+h_E $$
therefore
\begin{equation}\label{minw}
  w_\Omega \leq \min (1+h_N+h_S,  1+h_W+h_E).
 \end{equation}
 On the other hand, by convexity, the domain $\Omega$ contains the triangles joining each point $N,W,S,E$ to the two corresponding vertexes of
 the square: e.g. for the point $N$ these two vertexes are $(0,1)$ and $(1,1)$. Each such triangle having for area $h_N/2, h_W/2, h_S/2, h_E/2$
 we deduce the following lower bounds for the area of $\Omega$:
 \begin{equation}\label{areamin} 
 |\Omega| \geq 1+ (h_N+h_W + h_S + h_E )/2 \geq w_\Omega,
 \end{equation}
 where we used \eqref{minw} for the second inequality.
 Now, we use the following inequality for $\mu_1$ proved, for any planar domain $\Omega$ (not necessarily convex) in the paper
 \cite{HLL}:
 $$\mu_1(\Omega) \leq \frac{\pi^2 w^2_\Omega}{|\Omega|^2}$$
 where equality holds only for rectangles. Combining, \eqref{minw}, \eqref{areamin} and this last inequality, we have proved that
 $\mu_1(\Omega) \leq \pi^2 = \mu_1(Q)$. This shows that the square is the maximizer. Moreover, it is the only maximizer since for any rectangle
 different from the square, the inequality \eqref{minw} or the inequality \eqref{areamin}  must be strict.
\end{proof}
\begin{remark}
For the equilateral triangle we believe that the same property holds true. Indeed, it would follow from the following conjecture:\\
{\it for any planar convex domain $\Omega$ we have $P^2(\Omega) \mu_1(\Omega) \leq 16\pi^2$ with equality for the square
and the equilateral triangle} due to Laugesen-Polterovich-Siudeja, see the recent paper \cite{HLL} where this conjecture is proved
assuming that $\Omega$ has two axis of symmetry.  Assuming the conjecture is true, let $\Omega$ be a convex domain strictly
containing the equilateral triangle $T$, we have $P(T)<P(\Omega)$ therefore
$$P^2(T)\mu_1(\Omega) < P^2(\Omega)\mu_1(\Omega) \leq P^2(T)\mu_1(T)$$
proving the maximality of the equilateral triangle.
\end{remark}

\medskip
Since the disk (and the square) are not self-domains for $k=2$ for the exterior problem, one can wonder what the optimal domain looks like.
For the disk for example, we would expect some symmetry, thus a rather surprising result is the following:
\begin{proposition}
The optimal domain for the exterior problem with $k=2$ when the obstacle is the disk has not $j$-fold symmetry with $j\geq 3$
\end{proposition}
\begin{proof}
It is a classical property that a domain having $j$-fold symmetry with $j\geq 3$ must satisfy $\mu_1=\mu_2$, see for example
\cite[Lemma 4.1]{AB}. Therefore, applying Corollary \ref{coromu1mu2} we immediately get the conclusion.
\end{proof}
\section{The functional $J_k$}\label{secJk}
\subsection{Bounds for $\inf J_k$}
We recall that we can define, for any convex domain $D$ and for any integer $k$, the ``lack of monotonicity" by the formula
$$
J_k(D):= \frac{I_k(D)}{\mu_k(D)},    
$$
where $I_k(D)$ is defined by
$$
    I_k(D):= \inf\left\{ \mu_k(\Omega) \;:\; 
    \Omega \subset D,\; \Omega \text{ is a convex domain} \right\}.
$$
We are interested in the infimum of $J_k$ among all planar convex domains. We will denote it by
$$M_k=\inf\{J_k(D) : D \mbox { convex}, D\;\subset \mathbb{R}^2\}.$$
More precisely, if we can compute this infimum exactly for
$k=1$; we are just able to give bounds for $M_k$ in the general case.
Let us mention that recently, in \cite{Fre-Ken} P. Freitas and J. Kennedy introduced exactly the same number, but in any dimension $d$,
they denote it by $\alpha_{k,d}$ and then $M_k=\alpha_{k,2}$. In their paper, they also obtain the value of $M_1$ as in our Theorem
\ref{theoM1}, they also obtain the same upper bound for $M_k$ as in our Proposition \ref{prop_upper} but they do not give lower bounds 
for $M_k$ better than the trivial one (obtained by estimating $\mu_k(D)$ from below by $\pi^2/\mathrm{diam}^2(D)$).
In that sense, our Theorem \ref{theolowerMk} based on a new lower bound for any Neumann eigenvalue given in Theorem \ref{theolowermuk}
seems to be a real progress.

A first result, an easy consequence of Theorem \ref{theoI1} is the following:
\begin{theorem}\label{theoM1}
Let $M_k$ be defined as above. Then
$$M_1= \frac{\pi^2}{4(j_{0,1})^2}\simeq 0.427,$$
and the infimum in the definition of $M_1$ is not achieved.
\end{theorem}
\begin{proof}
We have already seen in Theorem \ref{theoI1} that
$$I_1(D):=\min_{\Omega \subset D \text{ , convex}} \mu_1(\Omega)=  \frac{\pi^2}{\mathrm{diam}^2(D)}.$$
Therefore the functional $J_1$ reduces to a simple form that can be bounded from below as follows,
$$J_1(D)= \frac{I_1(D)}{\mu_1(D)}= \frac{\pi^2}{\mathrm{diam}^2(D)\mu_1(D)}> \frac{\pi^2}{4j_{0,1}^2},$$
the last inequality directly comes from \cite{Kr}, \cite{HeMi}, where it has been proved that for all convex domains $D\subset \R^2$ it holds
$$
\mathrm{diam}^2(D)\mu_k(D)\leq (2 j_{0,1}+(k-1)\pi)^2.
$$
Moreover, it is proved in the above-mentioned papers and also in \cite{Cheng} for $\mu_1$ that this inequality is sharp, 
by considering a certain sequence of domains shrinking to a segment. 
For $\mu_1$, this is a sequence of isosceles triangles shrinking to its basis.
This gives the desired result.
\end{proof}

Let us now give an upper bound for $M_k$.
\begin{proposition}\label{prop_upper}
Let $M_k$ be defined as above. Then for all $k$ we have
$$M_k\leq \frac{k^2\pi^2}{(2j_{0,1}+(k-1)\pi)^2}<1.$$
\end{proposition}
\begin{proof}
We use the family of collapsing domains as defined in \cite{HeMi}, normalized with diameter $1$. For each $\eta>0$ there exists one of such $D$ (that is a very thin trapezoid) satisfying
$$\mu_k(D)\geq (2j_{0,1}+(k-1)\pi)^2(1-\eta).$$
Now for this fixed $D$ we consider as $\Omega$ a very thin rectangle  $R_\varepsilon$ of length $(1-\varepsilon)$, and of width so small that it fits into $D$. In that case we have  $\mu_k(R_\varepsilon)=k^2 \frac{\pi^2}{(1-\varepsilon)^2}$ as soon as the width is small enough. Since  $M_k$ and $I_k$ are infima we deduce that 
$$M_k\leq \frac{I_k(D)}{\mu_k(D)}\leq \frac{\mu_k(R_\varepsilon)}{\mu_k(D)}\leq \frac{k^2\pi^2}{(1-\varepsilon)^2 (2j_{0,1}+(k-1)\pi)^2(1-\eta)}.$$
Finally, since $\varepsilon$ and $\eta$ are arbitrary, we get
$$M_k\leq \frac{k^2\pi^2}{ (2j_{0,1}+(k-1)\pi)^2},$$
and the theorem is proved.
\end{proof}
\begin{remark}
We believe that the upper bound presented above is actually the true value of $M_k$. If we can prove that there is no existence
of a minimizer for $M_k$, this will follow by analyzing the behavior of a sequence of collapsing domains as done in \cite{HeMi}.
We present in the next subsection an iterative scheme that could also be used to check this property.
\end{remark}
Now we want to get a lower bound for $M_k$. This requires lower bounds for $\mu_k(\Omega)$ in terms of the diameter of $\Omega$.
Obviously, we have 
$\mu_k(\Omega) \geq \mu_1(\Omega) > \pi^2/\mathrm{diam}^2(\Omega)$ which is the lower bound used in \cite{Fre-Ken}, but we want
to improve it. Here is our result stated in dimension 2 (just after the proof of the theorem, we give the analogous inequality in dimension $d$).
\begin{theorem}\label{theolowermuk}
Let $\Omega$ be a bounded convex domain in the plane. Then we have the following lower bounds:
$$\mu_k(\Omega) \geq \frac{C_k}{\mathrm{diam}^2(\Omega)}.$$
The constant $C_k$ is computable for every $k$, moreover
$$C_2>\pi^2, \quad C_3>C_2,$$
and
$$C_k= \frac{[\sqrt{k}]^2 \pi^2}{2}\quad \forall k \geq 4$$
where $[\sqrt{k}]$ is the integer part (or the floor) of $\sqrt{k}$.
\end{theorem}
For the proof of this theorem, in the case $k\geq 4$, we will need the following elementary geometric lemma. We will denote by
$w(\Omega)$ the width of $\Omega$ in the direction orthogonal to the diameter (i.e. the minimal distance between two supporting planes
that are parallel to a diameter)
\begin{lemma} \label{box}
For any bounded convex domain $\Omega\subset \mathbb{R}^2$, there exists a rectangle of length $\mathrm{diam}(\Omega)$
and width $w(\Omega)$ that contains $\Omega$. In particular, there exists a 
square $Q$ with side length $\mathrm{diam}(\Omega)$ that contains $\Omega$.
\end{lemma}
\begin{proof} Let $\Omega$ be a convex domain and $R:=\mathrm{diam}(\Omega)<+\infty$ its diameter. Let $a,b \in \overline{\Omega}$ two endpoints of the diameter. Then it is easily seen that, see Fig. \ref{fig-square}:
$$\Omega\subset B(a,R)\cap B(b,R).$$
This means that $\Omega$ is contained in the strip delimited by two parallel lines: the tangent line to $\partial B(a,R)$ at $b$ and the tangent line to $\partial B(b,R)$ at $a$ as in Fig. \ref{fig-square}.
\begin{figure}[h]
\centering
\includegraphics[scale=0.2]{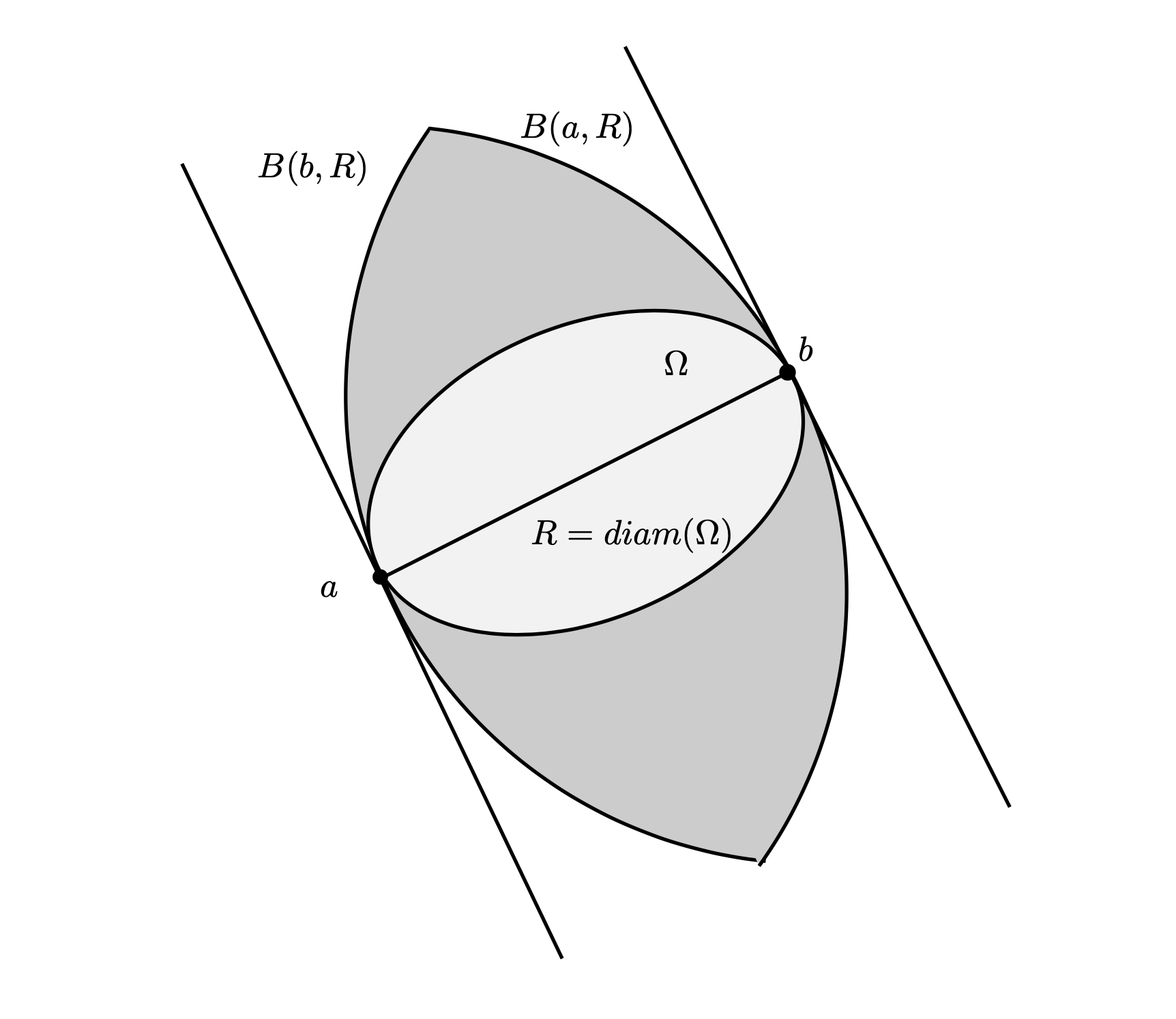}
\caption{Proof of Lemma  \ref{box}.} \label{fig-square}
\label{geometric construction}
\end{figure}
Moreover we see that $\Omega$ is inside the strip delimited by the two supporting lines parallel to the segment joining $a$ and $b$. Then, by definition of the width $w(\Omega)$, our domain is contained in a rectangle of length $\mathrm{diam}(\Omega)$
and width $w(\Omega)$. At last, since the width $w(\Omega)$ must be also smaller than $\mathrm{diam}(\Omega)$. it
follows that $\Omega$ lies inside a square of side $\mathrm{diam}(\Omega)$.
\end{proof}

\begin{remark}
The result of Lemma \ref{box} is optimal: indeed, if $\Omega$ is a Reuleaux triangle then it has constant width equal to its 
    diameter in any direction. 
    As a result, we can construct a square containing $\Omega$ with side of length exactly $\mathrm{diam}(\Omega)$.
\end{remark}

We are now in a position to prove Theorem \ref{theolowermuk}.
\begin{proof}
We start with the case $k\geq 4$. Let us set $N=[\sqrt{k}]$, then $N^2\leq k$.
We start by putting $\Omega$ inside a square $Q$ with side equal to $\mathrm{diam}(\Omega)$, as given by Lemma \ref{box}. Then 
we divide the square $Q$ in a collection $Q_j$ of $N^2$ cubes of side length $\mathrm{diam}(\Omega)/N$. 
Now we consider the partition of $\Omega$ defined by this grid: $\Omega_j=\Omega\cap Q_j$ and we use the Generalized Buser Lemma \ref{BuserLemma}:
$$\mu_{N^2}(\Omega) \geq \min_j \mu_1(\Omega_j) \geq \min_j \left(\pi^2/\mathrm{diam}^2(\Omega_j)\right).$$
But $\mathrm{diam}(\Omega_j) \leq \mathrm{diam}(Q_j)=\sqrt{2} \mathrm{diam}(\Omega)/N$, therefore we finally get
$$\mu_k(\Omega) \geq \mu_{N^2}(\Omega) > \frac{\pi^2 N^2}{2 \mathrm{diam}^2(\Omega)}$$
that is the desired result.

Let us now consider the case $k=2$. Buser's lemma alone cannot work because when we cut a convex domain in two parts, it is possible 
that each part has the same diameter as $\Omega$ itself. This leads us to split the class of convex domains into two sub-classes
$\mathcal{A}_f$ (f for flat) and $\mathcal{A}_r$ (r for round) defined by (we denote the diameter of $\Omega$ by $D(\Omega)$ here)
$$\mathcal{A}_f:=\{\Omega \; \hbox{s.t.}  \; w(\Omega)/D(\Omega) \leq \rho\}, \quad \mathcal{A}_r:=\{\Omega \; \hbox{s.t.} \;  w(\Omega)/D(\Omega) > \rho\}$$
where $\rho\in (0,1)$ is a threshold that we will choose at the end.

If $\Omega \in \mathcal{A}_f$, we consider the rectangle given by Lemma \ref{box} of dimensions $D(\Omega)$ and $w(\Omega)$
and we cut it in two rectangles $R_1,R_2$ along the longer side: their length is now $D(\Omega)/2$. Then we make a partition
of $\Omega$ as $\Omega_1 \cup \Omega_2$ where $\Omega_i=\Omega\cap R_i$ and we use Buser's Lemma \ref{BuserLemma}
that provides, since $\mathrm{diam}^2(\Omega_i) \leq w^2(\Omega) +D^2(\Omega)/4$
$$\mu_2(\Omega) \geq \min(\mu_1(\Omega_1),\mu_1(\Omega_2)) > \frac{\pi^2}{w^2(\Omega) +D^2(\Omega)/4}.$$
Using the property defining $\mathcal{A}_f$, it follows that 
\begin{equation}\label{minAf}
D^2(\Omega) \mu_2(\Omega) > \frac{\pi^2}{\rho^2 + \frac{1}{4}}.
\end{equation}
If $\Omega \in \mathcal{A}_r$, we use a recent result of D. Bucur and V.  Amato (to appear, private communication) that is a quantitative
improvement of the Payne-Weinberger inequality. They prove that there exists a constant $C$, that is computable in dimension 2, such that
for any convex domain it holds
\begin{equation}\label{buam}
D^2(\Omega) \mu_1(\Omega) \geq \pi^2 + C \frac{w^2(\Omega)}{D^2(\Omega)}.
\end{equation}
Note that this constant $C$ has a numerical value around 6.
Therefore, using in \eqref{buam} the property defining $\mathcal{A}_r$, we get
\begin{equation}\label{minAr}
D^2(\Omega) \mu_2(\Omega) \geq D^2(\Omega) \mu_1(\Omega) > \pi^2 + C \rho^2.
\end{equation}
Now we are interested in the minimum of the two values appearing in the right-hand side of Equations \eqref{minAf} and \eqref{minAr}
and we want to choose $\rho$ such that this minimal value is maximum. Since the two functions in $\rho$ are respectively
decreasing and increasing, we must choose $\rho$ such that
$$ \frac{\pi^2}{\rho^2 + \frac{1}{4}} = \pi^2 + C \rho^2.$$
Solving this quadratic equation in $\rho^2$, we immediately get
$$\rho^2=\frac{-(\pi^2+\frac{C}{4})+\sqrt{(\pi^2+\frac{C}{4})^2+3C\pi^2}}{2C}$$
that provides the universal lower bound
$$
D^2(\Omega) \mu_2(\Omega) > \frac{\pi^2}{2} \left(1+\sqrt{1+\frac{7C}{2\pi^2}+\frac{C^2}{16\pi^4}}\right) - \frac{C}{8}=:C_2 > \pi^2.
$$
%\textcolor{red}{The constant $C_2$ appearing above will be computed exactly when Dorin Bucur will write his paper with
%Vincenzo Amato and give us a possible value of $C$, numerically it seems that $C\simeq 6$};

Now, for $\mu_3$, we proceed exactly in the same way as for $\mu_2$. The only difference is that we cut now in three parts the rectangle
$(0,D(\Omega))\times (0,w(\Omega))$ along its long side and we will choose another threshold $\rho$ at the end. This leads us to the following
estimates:
\begin{itemize}
\item if $\Omega \in \mathcal{A}_f$, then
\begin{equation}\label{minAf2}
D^2(\Omega) \mu_3(\Omega) \geq \frac{\pi^2}{\rho^2 + \frac{1}{9}}.
\end{equation}
\item if  $\Omega \in \mathcal{A}_r$, then
\begin{equation}\label{minAr2}
D^2(\Omega) \mu_3(\Omega)  \geq  \pi^2 + C \rho^2.
\end{equation}
\end{itemize}
Choosing a value of $\rho$ making the two right-hand sides of \eqref{minAf2} and \eqref{minAr2} equal give, we finally obtain
$$
D^2(\Omega) \mu_3(\Omega) > \frac{\pi^2}{2} \left(1+\sqrt{1+\frac{34C}{9\pi^2}+\frac{C^2}{81\pi^4}}\right) - \frac{C}{18}=:C_3.
$$
This concludes the proof.
\end{proof}
\begin{remark}
Previous lower bounds for Neumann eigenvalues of convex domains in terms of their diameter were already known in any dimension.
For example, in References \cite{Grom, Sch-Yau, HKP} one can find bounds like 
$ \mu_k \geq c_d k^{2/d}/ (\mathrm{diam}^2(\Omega))$ where $c_d$ is a positive constant
depending only on the dimension $d$. Although all of these papers prove the estimate in a quantitative way, they do not give the explicit value of $c_d$ (one can certainly deduce it by following their proof). See also \cite{Funa} for a recent survey on these questions. In our Theorem \ref{theolowermuk} the constant $c_2$
is completely explicit (and quite simple).

In higher dimension $d\geq 3$, one can certainly follow the same strategy that we employed in Theorem \ref{theolowermuk}. This would lead to the
following lower bound
$$\mu_k(\Omega) \geq \frac{[\sqrt[d]{k}]^2 }{d} \frac{\pi^2}{\mathrm{diam}^2(\Omega)}$$
where $[\sqrt[d]{k}]$ is the integer part (or the floor) of $\sqrt[d]{k}$. This bound is comparable with the previous one
$ \mu_k \geq c_d k^{2/d}/ (\mathrm{diam}^2\Omega)$ with, here, an explicit constant $c_d$.
\end{remark}
%In higher dimension $d\geq 3$, we can proceed exactly in the same way to get a lower bound for $D^2(\Omega) \mu_k(\Omega)$. 
%We state the result without details
%\begin{theorem}\label{theolowermukd}
%Let $\Omega$ be a bounded convex domain in $\mathbb{R}^d$ Then we have the following lower bounds:
%
%If $\frac{[\sqrt[d]{k}]^2}{d} \leq 1$, this bound is not good and, in these cases we can always find a constant $C_{k,d}>1$
%such that
%$$\mu_k(\Omega) \geq \frac{C_{k,d} \pi^2}{\mathrm{diam}^2(\Omega)}.$$
%\end{theorem}
Using the lower bounds found in Theorem \ref{theolowermuk}, and using the upper bound $\mathrm{diam}^2(\Omega) \mu_k(\Omega)
< (2j_{0,1}+(k-1)\pi)^2$
found in \cite{Kr}, \cite{HeMi}, we immediately deduce
\begin{theorem}\label{theolowerMk}
We have the following lower bound for $M_k$:
$$M_k\geq \dfrac{C_k \pi^2}{(2j_{0,1}+(k-1)\pi)^2}$$
where  $C_k$ is the constant given in Theorem \ref{theolowermuk}.
\end{theorem}

\subsection{An iterative scheme}
Here we assume $k$ to be fixed.
Let $D_1$ be a given box, and $\Omega_1$ a solution of the {\it interior problem} for $D_1$. Then, let us
introduce $D_2$ a solution of the {\it exterior problem} for $\Omega_1$ and, by induction:
$\Omega_n$ a solution of the {\it interior problem} for $D_n$ and 
$D_{n+1}$ a solution of the {\it exterior problem} for $\Omega_n$. Then, we claim
\begin{theorem}\label{thm-iter}
The sequence $\{\mu_k(D_n)\}_{n\in \mathbb N}$ is increasing, the sequence $\{\mu_k(\Omega_n)\}_{n\in \mathbb N}$ is decreasing, therefore
the sequence $\{J_k(D_n)\}_{n\in \mathbb N}$ is decreasing.\\
Moreover, up to some subsequences,  the sequences of convex sets $\{\Omega_n\}_{n\in \mathbb N}$ and $\{D_n\}_{n\in \mathbb N}$ 
either converge in the Hausdorff sense to a pair $(\Omega^\ast,D^\ast)$ 
that is stationary for this construction (i.e.  $\Omega^\ast$ is solution of the interior problem for $D^\ast$ and $D^\ast$ is 
solution of the exterior problem for $\Omega^\ast$) 
or both converge to a segment.
\end{theorem}
Therefore, if there exists a minimizing pair for $J_k$, this scheme could be a good way to get it. 
We can qualify it as a descent algorithm since $J_k$ decreases along the sequence.
On the other hand,
if the sequences collapse to a segment, this would prove non-existence and give the value of $\inf J_k$
as explained before.
\begin{proof}[Proof of Theorem \ref{thm-iter}]
Let $k$ be fixed. Let $\{\Omega_n\}_n$ and $\{D_n\}_n$ be the two sequences of convex sets defined by recursion as above. 
Exploiting the optimality of $\Omega_n$ and $\Omega_{n+1}$, we infer that 
$$
J_k(D_n)=\frac{I_k(D_n)}{\mu_k(D_n)} = \frac{\mu_k(\Omega_n)}{\mu_k(D_n)}, \quad 
J_{k}(D_{n+1}) =\frac{I_{k}(D_{n+1})}{\mu_k(D_{n+1})} =  \frac{\mu_k(\Omega_{n+1})}{\mu_k(D_{n+1})}.
$$
By construction, we have that both $D_n$ and $D_{n+1}$ contain $\Omega_n$, and between the two, $D_{n+1}$ has greater $\mu_k$, in view of its optimality for the exterior problem on $\Omega_n$. In formulas, we have
$$
\mu_k(D_{n+1}) \geq \mu_k(D_n).
$$
Similarly, both $\Omega_n$ and $\Omega_{n+1}$ are contained into $D_{n+1}$, and between the two, $\Omega_{n+1}$ has lower $\mu_k$, in view of its optimality for interior problem on $D_{n+1}$. In formulas, we have
$$
\mu_k(\Omega_{n+1}) \leq \mu_k(\Omega_n). 
$$
By combining these inequalities with the expressions of $J_k(D_n)$ and $J_k(D_{n+1})$, we deduce the desired monotonicity, namely $J_k(D_n) \geq J_k(D_{n+1})$.

Now, let us check that the two sequences $\Omega_n$ and $D_n$ have bounded diameters. Assume that (for a subsequence) the diameter of $D_n$ is not bounded; by the inequality
$$\mu_k(D_n) \leq (2j_{0,1}+(k-1)\pi)^2/\mathrm{diam}^2 (D_n)$$
this would imply that $\mu_k(D_n)$ goes to zero; in contradiction with the fact that the sequence $\mu_k(D_n)$ is increasing.
Then the diameters of the $D_n$ are uniformly bounded. Moreover, since $\Omega_n \subset D_n$ this is also the case for the $\Omega_n$.
Finally, the sequences having bounded diameter, the alternative on the convergence of these sequences of convex domains is classical.

Now we show that 
\begin{align}\label{stationary1}
\mu_k(\Omega^{\ast})=&\inf \{\mu_k(\Omega) , \Omega\subseteq D^{\ast}, \Omega\text{ convex}\}
\end{align}
For a domain $\Omega$ in $\mathbb{R}^2$, $x\in \mathbb{R}^2$, and $r>0$ we set 
$$
r\Omega + x:= \{ry+x: y\in \Omega\}.
$$
since $D_n$ converges to $D_{\ast}$ for the Hausdorff convergence (of convex sets) we have, see \cite[Chapter 2]{HenPie}
$$
D_{\ast}\subseteq (1+\varepsilon)(D_n-x_n)+x_n
$$
for some $x_n\in \mathbb{R}^2$ and for sufficiently large $n$. We thus get
$$
\frac{1}{1+\varepsilon}(\Omega-x_n)+x_n\subseteq D_n,
$$
which gives 
$$
\mu_k(\Omega_n)\leq \mu_k\Big(\frac{1}{1+\varepsilon}(\Omega-x_n)+x_n\Big)=(1+\varepsilon)^2\mu_k(\Omega)
$$
by the definition of $\Omega_n$.
Letting $n\to \infty$ and then $\varepsilon \to 0$ we have $\mu_k(\Omega_{\ast})\leq \mu_k(\Omega)$, which shows (\ref{stationary1}).

Next we show the equality 
\begin{align}\label{stationary2}
    \mu_k(D_{\ast})=\ &\sup\{\mu_k(D'), \Omega_{\ast}\subseteq D', D'\text{ convex}\}.
\end{align}
Let $D'$ be convex such that $\Omega_{\ast}\subseteq D'$. For any $\varepsilon>0$ since $\Omega_n$ converges to $\Omega_{\ast}$ and each $\Omega_{\ast}$ is not a segment we have
 $$
 \Omega_n\subseteq (1+\varepsilon)(\Omega_{\ast}-y_n)+y_n
 $$
 for some $y_n\in \mathbb{R}^2$ and for sufficiently large $n$. This implies
 $$
 \Omega_n\subseteq (1+\varepsilon)(D'-y_n)+y_n
 $$
 and by the definition of $D_{n+1}$ we obtain
 $$
 \mu_k(D_{n+1})\geq \mu_k( (1+\varepsilon)(D'-y_n)+y_n)=\frac{1}{(1+\varepsilon)^2}\mu_k(D').
 $$
 Letting $n\to \infty$ and then $\varepsilon \to 0$ we have $\mu_k(D_{\ast})\geq \mu_k(D')$, which shows (\ref{stationary2}). This completes the proof.
\end{proof}
\begin{remark} If we can prove that a geometric quantity, like the minimal width, goes to zero under this
algorithm, this would also a give a direct way to prove non-existence of a minimizer for $\inf J_k$.
\end{remark}

\section{Some numerical illustrations}\label{secnum}

In this section we present a numerical optimization scheme for the interior problem and the exterior problem.

\medskip

{\bf Theoretical description of shapes.} Here we briefly introduce the representation of admissible shapes based on support functions. For more details about this standard approach, we refer to \cite{AnBo} and the references therein.

We identify a planar convex set $\Omega$ with its support function $f_\Omega$ defined as follows:
$$
f_\Omega: \mathbb S^1 \to \mathbb R, \quad f_\Omega(\theta):= \sup_{x \in \Omega} \left(x\cdot (\cos \theta, \sin \theta)\right),
$$
where, without loss of generality, we have assumed the origin to be an interior point of $\Omega$. The support function, in turn, being $2\pi$-periodic, can be identified with the collection of Fourier coefficients $a_0, \{a_k, b_k\}_{k\in \mathbb N}$:
$$
f_\Omega(\theta)=a_0 + \sum_{k\geq 1} [ a_k \cos(k\theta) + b_k \sin(k \theta)].
$$
Following our notation for the interior and exterior problem, we denote by $f_\Omega$ the admissible support functions for the interior problem and by $f_D$ the admissible support functions for the exterior problem. 

Let us now pass to the description of the constraints. Convexity is encoded by the following inequality, intended in the sense of distributions:
\begin{equation}\label{constraint1}
f_\Omega(\theta) + f_\Omega''(\theta) \geq 0, \quad f_D(\theta) + f_D''(\theta) \geq 0.
\end{equation}
In terms of the Fourier coefficients (of $f_\Omega$ or $f_D$), this reads
\begin{equation}\label{constraint1F}
a_0 + \sum_{k \geq 1}(1-k^2)[a_k \cos(k\theta) + b_k \sin(k \theta)] \geq 0.
\end{equation}
On the other hand, the set inclusion $\Omega_1\subset \Omega_2$ is equivalent to the ordering $f_{\Omega_1} \leq f_{\Omega_2}$ between the support functions. Therefore, in the interior and exterior problem we will impose 
\begin{equation}\label{constraint2}
f_{\Omega} \leq f_D, \quad f_D \geq f_\omega,
\end{equation}
where $D$ is the bounding box and $\omega$ is the obstacle. In terms of the Fourier coefficients, we have: for the interior problem
\begin{equation}\label{constraint2F}
a_0 + \sum_{k\geq 1} [ a_k \cos(k\theta) + b_k \sin(k \theta)] \leq \widehat{a}_0 + \sum_{k\geq 1} [ \widehat{a}_k \cos(k\theta) + \widehat{b}_k \sin(k \theta)],
\end{equation}
and, for the exterior problem,
\begin{equation}\label{constraint2Fbis}
a_0 + \sum_{k\geq 1} [ a_k \cos(k\theta) + b_k \sin(k \theta)] \geq \widehat{a}_0 + \sum_{k\geq 1} [ \widehat{a}_k \cos(k\theta) + \widehat{b}_k \sin(k \theta)],
\end{equation}
where $\widehat{a}_0, \{\widehat{a}_k, \widehat{b}_k\}_{k \in \mathbb N}$ denote, for the interior and exterior problem, the Fourier coefficients of $f_D$ or $f_\omega$, respectively.

\medskip

{\bf Numerical description of shapes.} For the numerical optimization, we need to work in a finite dimensional space: to this aim, following \cite{AnBo}, we adopt two alternative strategies. For the benefit of the reader, we recall here the main ideas.

\medskip

In the first strategy we approximate shapes by considering the truncated Fourier series of the support function at some index $N$. The unknown 
of the problem is then a vector
$$
F=(a_0, a_1, \ldots, a_N, b_1, \ldots, b_N) \in \mathbb R^{2N+1}.
$$
The two constraints are imposed on a discrete set of points, instead of the whole interval $[0,2\pi]$. We fix $M\in \mathbb N$ and we consider 
$$
\theta_m:= m \frac{2\pi}{M}, \quad m = 1,\ldots, M.
$$
Imposing \eqref{constraint1F}  on every $\theta_m$, we get
$$
a_0 + \sum_{k = 1}^N(1-k^2)[a_k \cos(k\theta_m) + b_k \sin(k \theta_m)] \geq 0 \quad \forall m = 1, \ldots, M.
$$
Similarly, imposing \eqref{constraint2F}-\eqref{constraint2Fbis} on every $\theta_m$, we get
$$
a_0 + \sum_{k= 1}^N [ a_k \cos(k\theta_m) + b_k \sin(k \theta_m)] \leq \widehat{a}_0 + \sum_{k= 1}^N [ \widehat{a}_k \cos(k\theta_m) + \widehat{b}_k \sin(k \theta_m)]
$$
for the interior problem and the reverse inequality for the exterior problem. Here $m$ runs from 1 to $M$.

Now we notice that both inequalities are linear in $F$ and can be rewritten in a more convenient way as
$$
AF \leq B,
$$
where $A$ is a $2M\times (2N+1)$-matrix and $B$ is a $2M$-vector. Let us write the components of $A$ and $B$. In the following, the indexes $m$ and $k$ run from 1 to $M$ and from 1 to $N$, respectively. For the interior problem, we have
$$
\left\{
\begin{array}{lll}
A_{m,1}& =-1 
\\
A_{m,1+k} &= - (1-k^2)\cos(k\theta_m)
\\
A_{m,1+N+k} &= - (1-k^2)\sin(k\theta_m)
\\
A_{M+m,1} & =  1
\\
A_{M+m,1+k} & = \cos(k \theta_m)
\\
A_{M+m,1+N+k} & = \sin(k \theta_m)
\end{array}
\right.
$$
and 
$$
\left\{
\begin{array}{ll}
B_m& = 0
\\
B_{M+m}&=
\widehat{a}_0 + \sum_{k= 1}^N [ \widehat{a}_k \cos(k\theta_m) + \widehat{b}_k \sin(k \theta_m)].
\end{array}
\right. 
$$
For the exterior problem, we have
$$
\left\{
\begin{array}{lll}
A_{m,1}& =-1 
\\
A_{m,1+k} &= - (1-k^2)\cos(k\theta_m)
\\
A_{m,1+N+k} &= - (1-k^2)\sin(k\theta_m)
\\
A_{M+m,1} & =  -1
\\
A_{M+m,1+k} & = -\cos(k \theta_m)
\\
A_{M+m,1+N+k} & = -\sin(k \theta_m)
\end{array}
\right.
$$
and 
$$
\left\{
\begin{array}{ll}
B_m& = 0
\\
B_{M+m}&= -
\widehat{a}_0 - \sum_{k= 1}^N [ \widehat{a}_k \cos(k\theta_m) + \widehat{b}_k \sin(k \theta_m)].
\end{array}
\right. 
$$

\medskip

The second strategy consists in considering a piece-wise affine approximation of $f_\Omega$ (for the interior problem) or $f_D$ (for the exterior problem): given $M\in \mathbb N$, the unknown is the vector
$$
F=(f_1, \ldots, f_M)\in \mathbb R^M,
$$
where $f_i$ represents $f_\Omega(i \tau )$ (for the interior problem) or  $f_D(i \tau)$ (for the exterior problem), with $\tau:=2\pi/M$. Here and in the following lines, the index $i$ will run from $1$ to $M$. 
Taking the approximation of derivatives by finite differences, we write the convexity constraint \eqref{constraint1} as
$$
\forall i \quad 
f_i + \frac{ f_{i+1} - 2 f_i + f_{i-1}  }{\tau^2}  \geq 0, 
$$
where by periodicity we set $f_{-1}:=f_M$ and $f_{M+1}:=f_1$.
On the other hand, the constraint \eqref{constraint2} simply reads
$$
f_i \leq f_{D}( i\tau )\quad \hbox{or}\quad f_i \geq f_{\omega}(i \tau)
$$
for the interior or exterior problem, respectively.
As in the previous strategy, the two constraints can be rewritten in a more convenient way as $AF\leq B$, where $A$ is a $2M\times M$ matrix and $B$ is a $2M$ vector. The matrix $A$ is made of 2 blocks $M\times M$ aligned vertically. The first sub-matrix (above) is ``almost" tridiagonal, in the following sense: it has elements $-1+2/\tau^2$ on the main diagonal, elements $-1/\tau^2$ in the upper/lower diagonal and, due to the periodicity of the support function, also $A_{1,M}$ and $A_{M,1}$ are equal to $-1/\tau^2$; all the other elements are zero. The second (below) sub-matrix of $A$ is the identity matrix for the interior problem and it is minus the identity matrix for the exterior problem. 
The vector $B$ has two blocks, too: the former is the zero M-vector, namely $B_i=0$; the latter is $B_{M+i}=f_{D}(i \tau)$ for the interior problem and  $B_{M+i}= - f_{\omega}(i \tau)$ for the exterior problem.

\medskip

{\bf Optimization scheme for the interior problem.} Given $k \in \mathbb N$, we perform the following optimization: 
$$
\min\{\mu_k(\Omega_F)\ :\ F\in \mathbb R^{d}, \quad AF\leq B \}.
$$
Where $F$ is either the vector representing the first $d$ Fourier coefficients of $f_\Omega$ or the discretization of $f_\Omega$ at the points $i 2\pi/d$, $i=1, \ldots, d$. The matrix $A$ and the vector $B$ are constructed accordingly, following the procedure described in the previous paragraph. 

Here $\Omega_F$ is the convex shape associated to $f_\Omega$. The computation of $\mu_k$ is done using the Matlab function \texttt{solvepdeeig}. The optimization is run using the \texttt{fmincon} routine of Matlab, with linear inequality constraints, and taking a random starting point $F_0$.

Let us now present some examples, in which we take $D$ to be the square or the disk, and $k=1, 2, 3, 4$. Our numerical optimization goes in the same direction of the theoretical study performed in Section \ref{secsquaredisk}, in the following sense:
\begin{itemize}
\item for $k=1$, numerics suggest that the best shape is a segment realizing the diameter (see Fig. \ref{fig-mu1disk});
\item for $D=$ disk and $k=2,4$ or $D=$ square and $k=2,3$, numerics suggest that $D$ should be a self-domain;
\item for $D=$ disk and $k=3$ we find a better domain than the disk;
\item for $D=$ square and $k=4$ we find a better domain than the square.
\end{itemize}

\begin{center}
\begin{figure}[h]
\includegraphics[width=5cm]{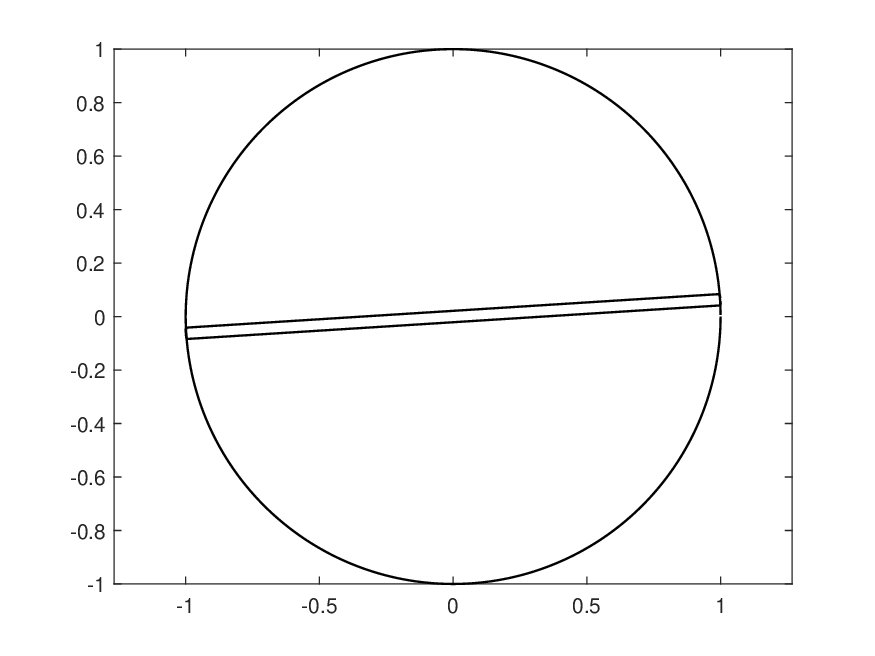}
\caption{Optimal shape for $k=1$ in the unit disk with $\mu_1 = 2.468$. Optimization algorithm: Strategy 2, with 50 points, and constraint on the minimal height of the mesh $H_{min}=0.001$.}\label{fig-mu1disk}
\end{figure}
\end{center}

Let us present more in detail the two last items. We recall that when $D$ is the unit disk, $\mu_3(D)=j_{2,1}^2\simeq 9.33$. Following Strategy 1, we find a shape with $\mu_3 = 8.86$ (see Fig. \ref{fig-mu3disk}-left), whereas following Strategy 2, we find a shape with $\mu_3=8.516$ (see Fig. \ref{fig-mu3disk}-right). Both strategies allow to confirm that the disk is not a self-domain for the interior problem with $k=3$. 

\begin{center}
\begin{figure}[h]
\includegraphics[width=5cm]{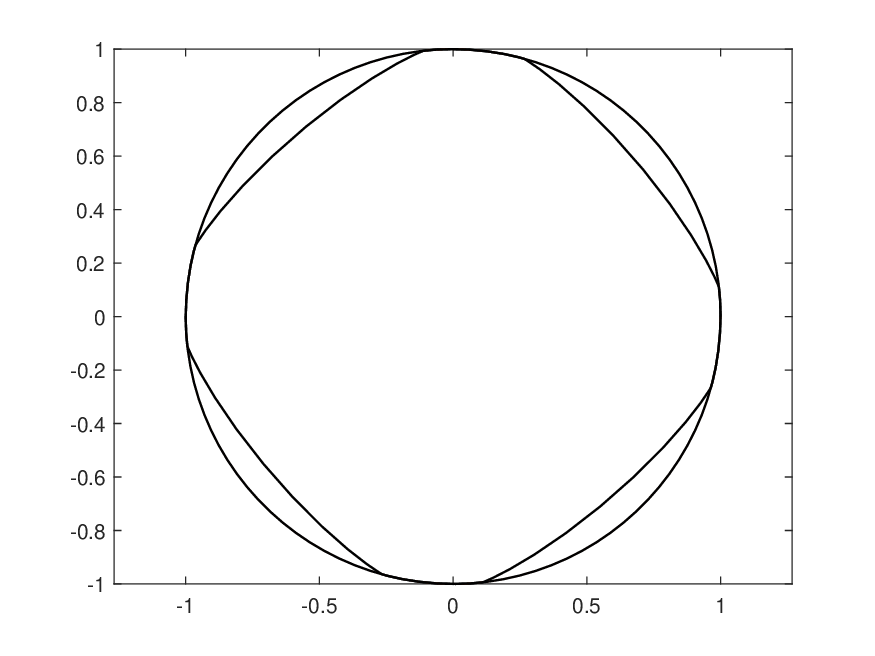}\quad \includegraphics[width=5cm]{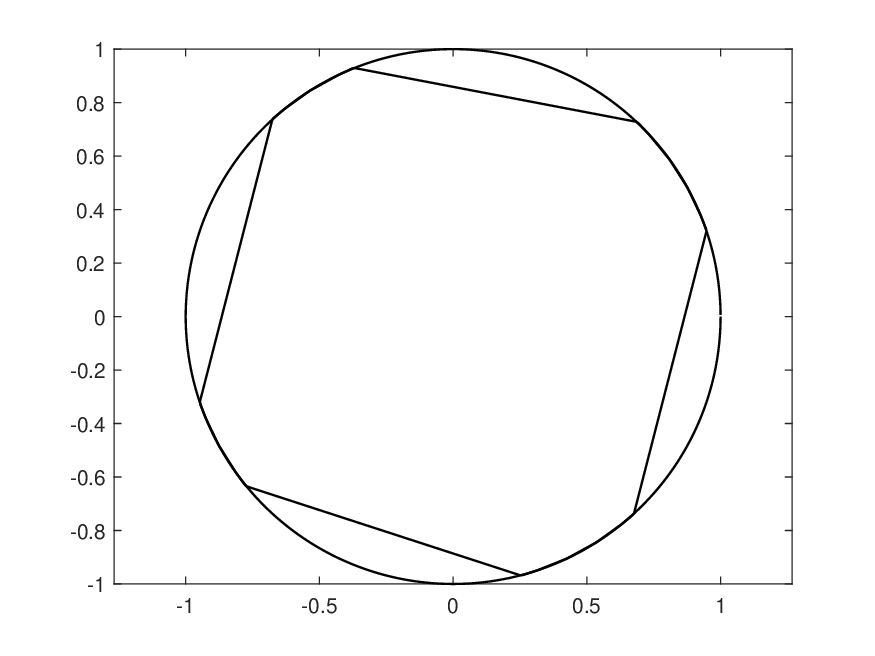}
\caption{Left: shape with $\mu_3= 8.86$ obtained using Strategy 1 with 17 Fourier coefficients. Right: shape with $\mu_3=8.51$ obtained using Strategy 2 with 50 points.}\label{fig-mu3disk}
\end{figure}
\end{center}

Numerics seem to suggest that the optimal shape should be the intersection of the disk $D$ and a larger square. These shapes can be described by one parameter, the angle $2\theta$ of each circular sector, with $\theta \in [0,\pi/4]$. Let us denote by $\Omega_\theta$ the shape associated to $\theta$. Then $\Omega_0$ is the square inscribed into the disk $D$, $\Omega_{\pi/4}=D$ (intersection of the disk and the circumscribed square). A numerical optimization gives $\theta_{opt}=0.21$ with $\mu_3(\Omega_{\theta_{opt}}) = 8.47$, see Fig. \ref{fig-mu3disk2}.

\begin{center}
\begin{figure}[h]
\includegraphics[width=5cm]{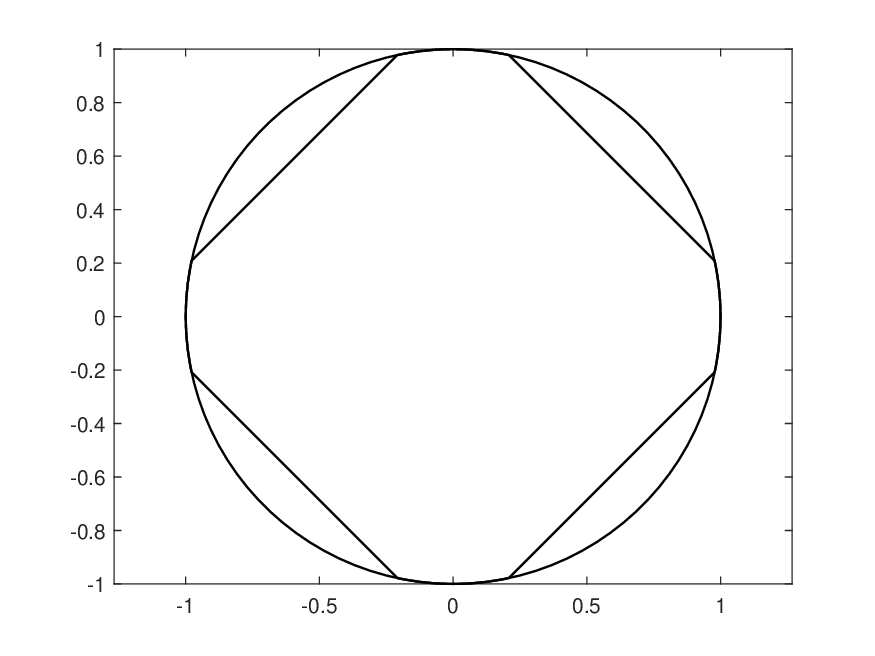}\quad \includegraphics[width=5cm]{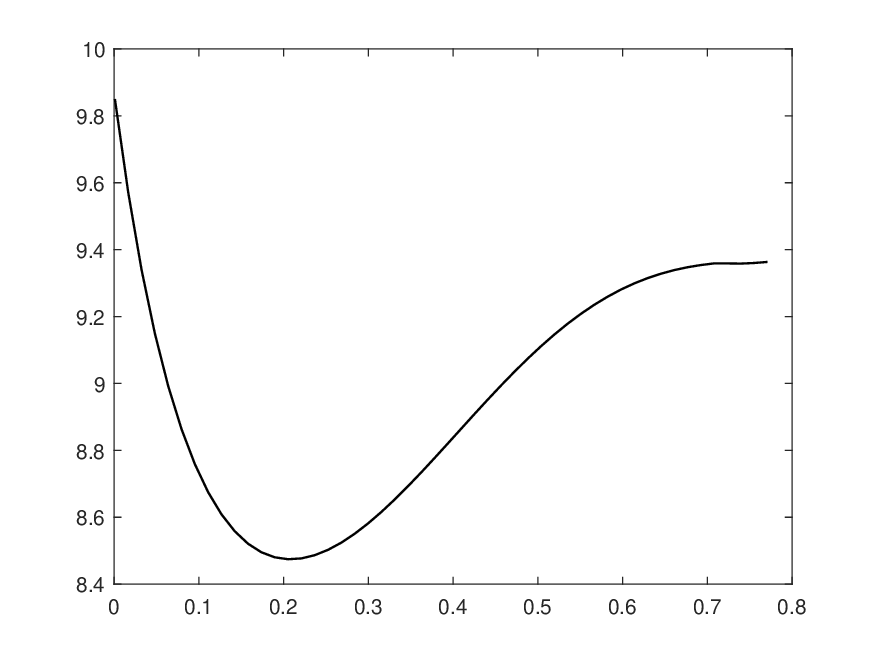}
\caption{Left: numerical optimizer associated to $\theta=0.21$. Right: plot of the function $\theta\mapsto \mu_3(\Omega_\theta)$.}\label{fig-mu3disk2}
\end{figure}
\end{center}

Let now $D$ be the square $[-1,1]^2$ and $k=4$. We recall that $\mu_4(D)=\pi^2 \simeq 9.86$. As before, numerics suggest that the optimal shape should be a polygon inside $D$, more precisely, an octagon (see Fig. \ref{fig-mu4square}-left). When we restrict to octagons, numeric optimization provides many local minimizers, all of them have $\mu_4$ striclty less than the square $D$. In Fig. \ref{fig-mu4square}-right an example.

\begin{center}
\begin{figure}[h]
\includegraphics[width=5cm]{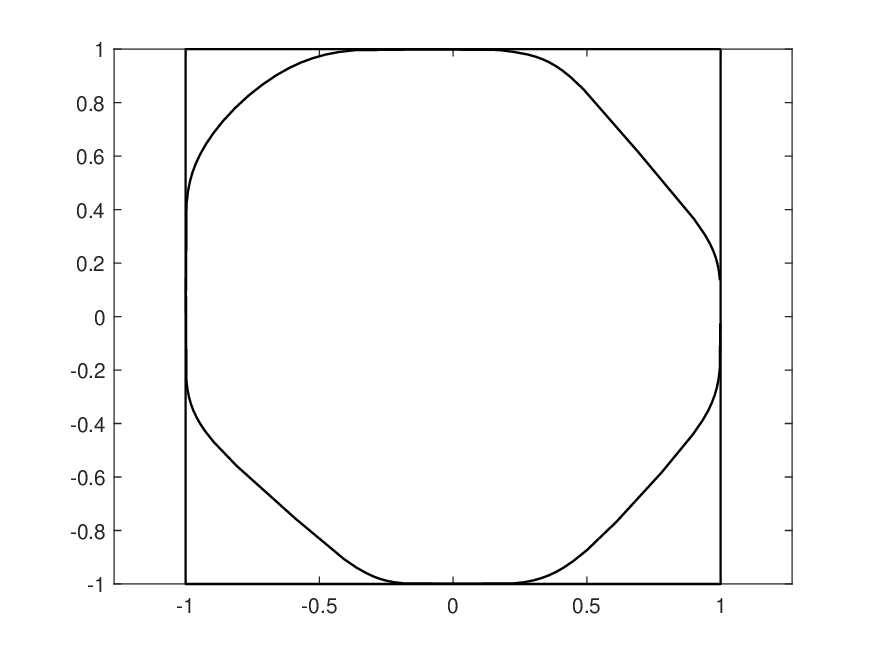}\quad \includegraphics[width=5cm]{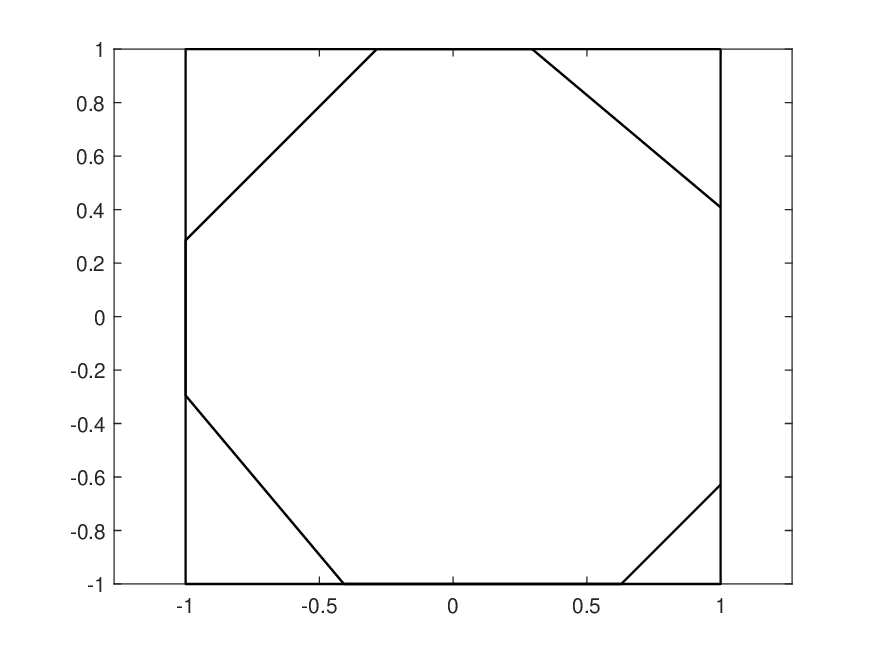}
\caption{Left: shape with $\mu_4=8.95$ obtained using Strategy 2 with 80 points. Right: a local minimizer with $\mu_4=8.74$ among octagons.}\label{fig-mu4square}
\end{figure}
\end{center}

{\bf Optimization scheme for the exterior problem.} We follow the same idea used for the interior problem, with the proper modifications: to maximize $\mu_k$ we solve
$$
\min\{-\mu_k(D_F)\ :\ F\in \mathbb R^{d}, \quad AF\leq B \},
$$
where $A$ and $B$ are the matrix and the vector defined above, and $d$ is either the number of Fourier coefficients of $f_D$ that we are considering (Strategy 1), or the number of discretization points of the variable of $f_D$ (Strategy 2).

As in the previous paragraph, we consider two obstacles, the disk and the square, and the first 4 indexes $k$. Numerics seem to confirm the theoretical study performed in Section \ref{secsquaredisk}, in the following sense:
\begin{itemize}
\item for $\omega=$ disk and $k=1, 3$ or $\omega=$ square and $k=1, 3, 4$, numerics suggest that $\omega$ should be a self-domain;
\item for $\omega=$ disk and $k=2,4$ we find better domains than the disk;
\item for $\omega=$ square and $k=2$ we find a better domain than the square.
\end{itemize}

Let us explain in more detail the two last items. Let us start with $\omega$ the unit disk, for which $\mu_2(\omega)=j_{1,1}^2 \simeq 3.39$ and $\mu_4(\omega)=j_{2,1}^2 \simeq 9.33$. Numerics allow to find better shapes, searched in  a particular class of shapes: convex envelopes of the disk and 2 points for $k=2$, and of the disk and 4 points for $k=4$. The values are specified in Fig. \ref{fig-mu2disk} and Fig. \ref{fig-mu4disk}.

\begin{center}
\begin{figure}[h]
\includegraphics[width=5cm]{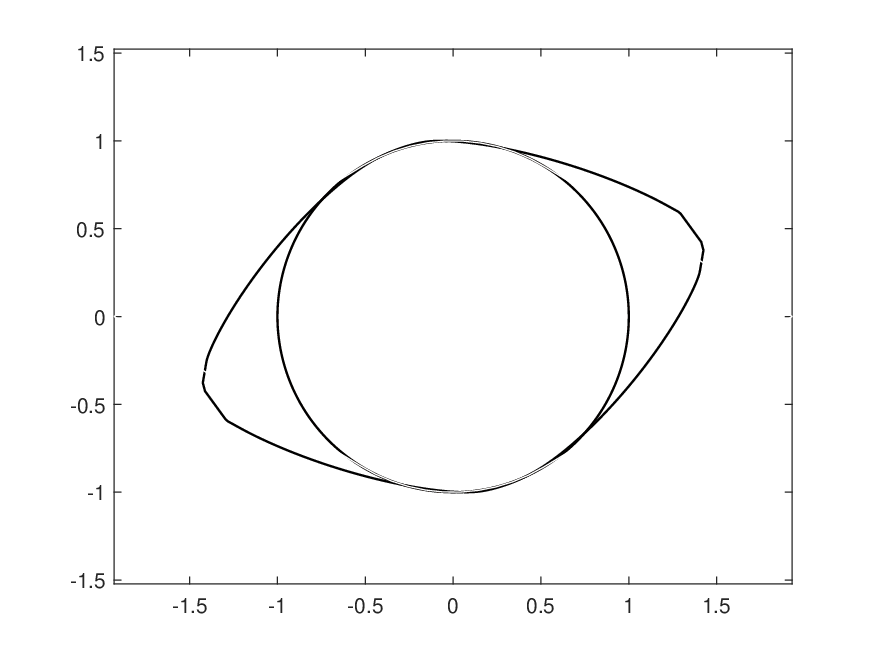}\quad \includegraphics[width=5cm]{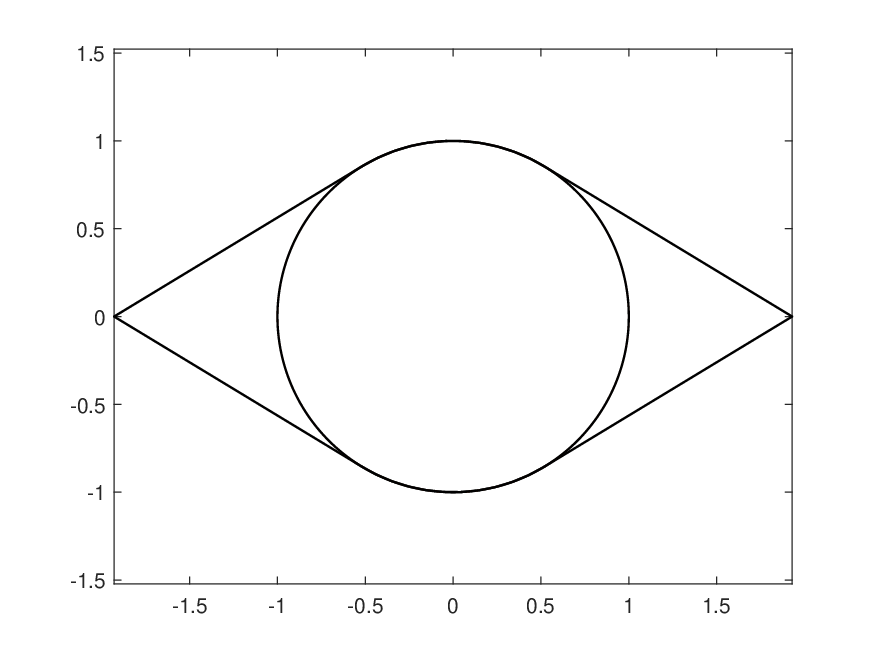}
\caption{Left: shape with $\mu_2=3.59$ obtained using Strategy 1, with 21 Fourier coefficients. Right: candidate optimizer with $\mu_2=3.61$ searched among convex envelopes of the disk and 2 points, obtained with a pair of antipodal points with distance 1.93 from the origin.}\label{fig-mu2disk}
\end{figure}
\end{center}
\begin{center}
\begin{figure}[h]
\includegraphics[width=5cm]{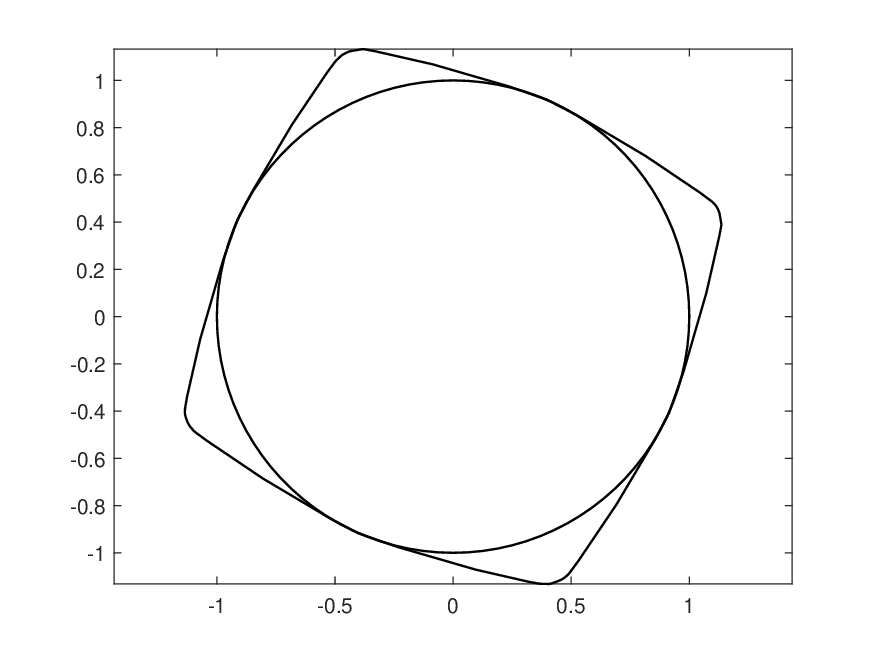}\quad \includegraphics[width=5cm]{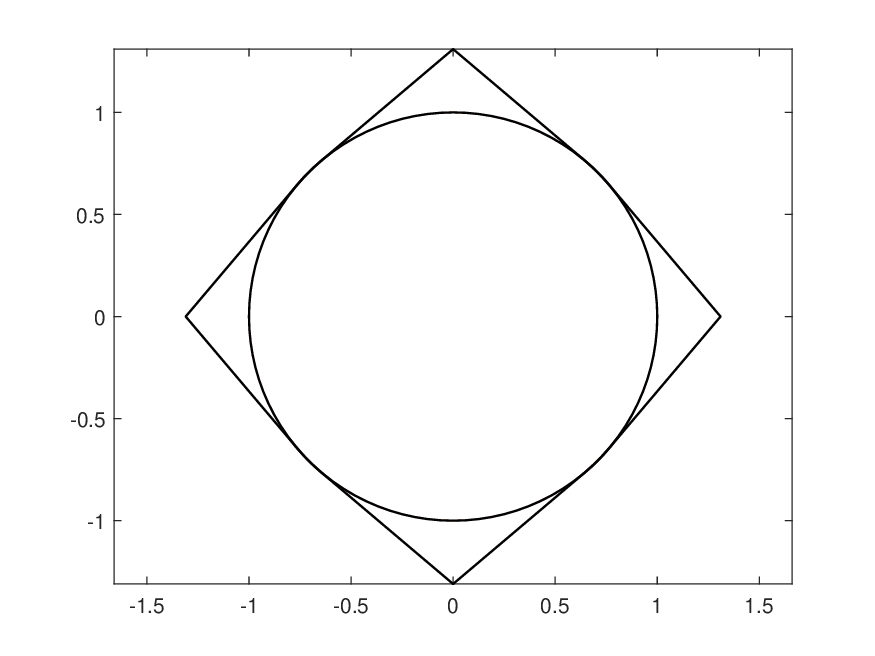}
\caption{Left: shape with $\mu_4=9.87$ obtained using Strategy 1, with 59 Fourier coefficients. Right: candidate maximizer with $\mu_4=9.94$ searched among convex envelopes of the disk and 4 points, obtained with 2 pairs of antipodal points with distance 1.31 from the origin.}\label{fig-mu4disk}
\end{figure}
\end{center}

We conclude with the case of the square $\omega=[-1,1]^2$ and $k=2$. Recalling that $\mu_2(\omega)=\pi^2/4 \simeq 2.46$, we show that the square is not a self-domain for the exterior problem: we first find a shape $D \supset \omega$ with $\mu_2=2.84$ by following Strategy 2 and then, working among hexagons which are convex envelopes of the square and a pair of points, we find another (better) candidate with $\mu_2=2.88$.

\begin{center}
\begin{figure}[h]
\includegraphics[width=5cm]{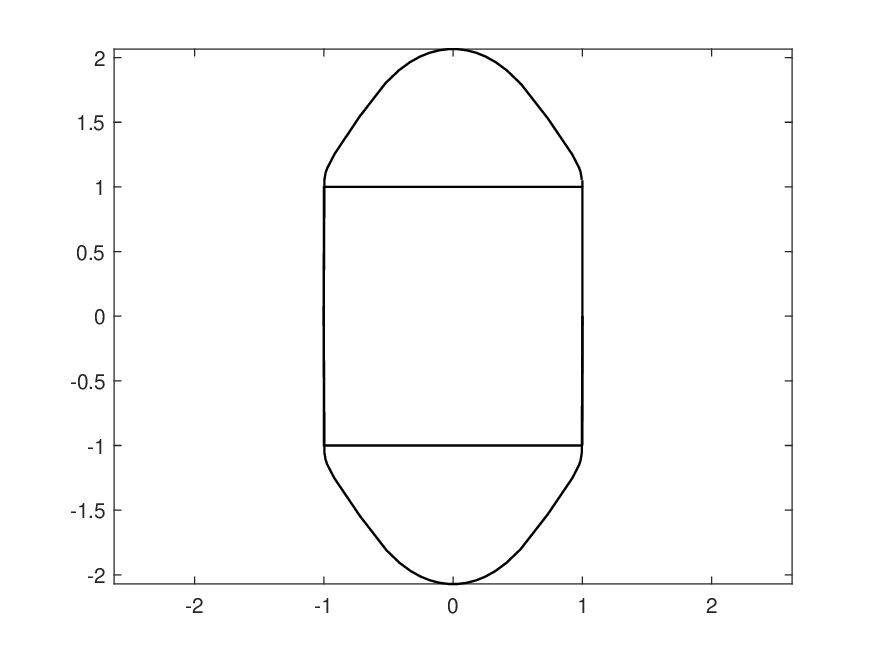}\quad \includegraphics[width=5cm]{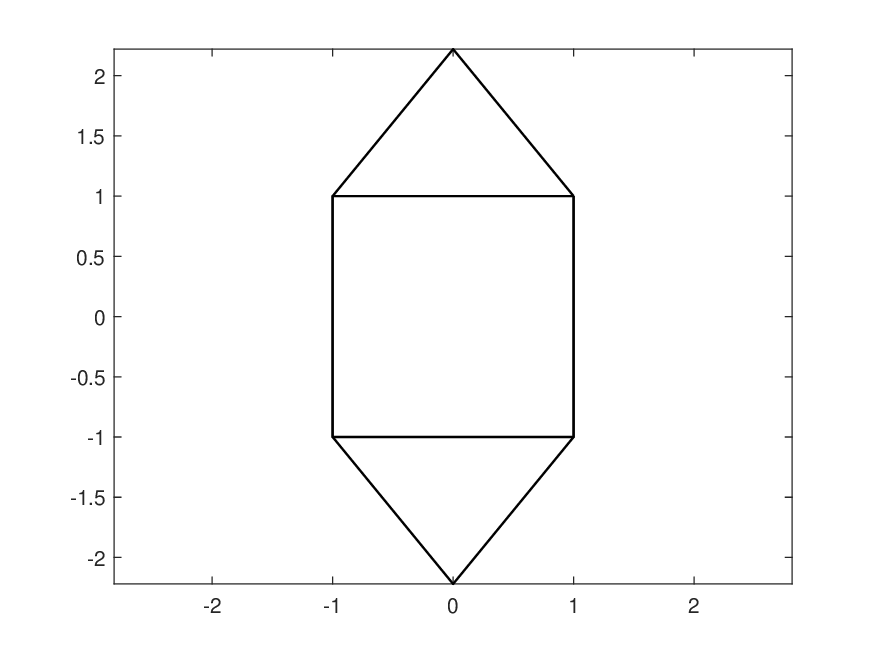}
\caption{Left: shape with $\mu_2=2.84$, obtained using Strategy 2 with 50 points. Right: candidate maximizer with $\mu_2=2.88$ searched among convex envelopes of the square and 2 points, obtained with a pair of antipodal points with distance 2.22 from the origin.}
\end{figure}
\end{center}
%%%%%%%%%%%%%%%%%%%%%%%%%%%%%%%%%%%%%%%%%%%%%%%%

\vspace{5mm}
{\bf Acknowledgements}: 
The six authors want to thank first the University of Lorraine and the University of Tohoku for a Grant that allow them to travel in Japan and
in France in 2023 making this joint work possible.
For the French part, this work has also been supported by the project ANR-18-CE40-0013 SHAPO financed by the French Agence Nationale de la Recherche (ANR).  
L. Cavallina was partially supported by JSPS KAKENHI Grant Numbers JP21KK0044 and JP22K13935, JP23H04459.
K. Funano is supported by JSPS KAKENHI Grant Number JP17K14179.
 I. Lucardesi is member of the Italian research group GNAMPA of Istituto Nazionale di Alta Matematica (INdAM) and her work has been partially supported by the the INdAM-GNAMPA project 2023 “Esistenza e proprietà fini di forme ottime" n. CUP-E53C22001930001.  
 S. Sakaguchi has been partially supported by JSPS KAKENHI Grant Numbers JP18H01126 and JP22K03381.
 
\appendix
%%%%%%%%%%%%%%%%%%%%%%%%%%%%%%%%%%%%%%%%%%%%%%%%%
\section{Proof of the generalized Buser's estimate}

For the sake of completeness, here is the proof in Buser's book \cite[section 8.2.1]{Bus}, revisited in the Euclidean space $\R^N$ and slightly generalized. We recall the statement here below.

\begin{lemma}\label{buslemapp}
(Generalized Buser Lemma) For any bounded domain $\Omega$ that is decomposed into a $j$-partition 
$\Omega_1,\ldots,\Omega_j$ and for any decomposition of the integer $k$ as the sum of $j$ positive integers: $k=k_1+\ldots + k_j$, we have 
\begin{eqnarray}
\mu_k(\Omega)\geq \min_i \mu_{k_i}(\Omega_i).\label{busineq1}
\end{eqnarray}
Moreover,   if  the inequality \eqref{busineq1} is an equality, then $\mu_{k_i}(\Omega_i)=\mu_{k}(\Omega)$ for all $i$ and there exists an eigenfunction $\varphi \in H^1(\Omega)$ associated with $\mu_k(\Omega)$ whose restriction to each $\Omega_i$ is also a Neumann eigenfunction for $\mu_{k_i}(\Omega_i)$. 
\end{lemma}

\begin{remark} The equality case in Lemma \ref{BuserLemma} occurs for instance when $\Omega$ is a square divided into two rectangles $\Omega_1$, $\Omega_2$ in such a way that $\mu_1(\Omega)=\mu_2(\Omega)=\mu_1(\Omega_1)=\mu_1(\Omega_2)$.
\end{remark}

\begin{proof}[Proof of Lemma \ref{buslemapp}]  Let us consider an $L^2$-orthonormal basis $f_0^i,f_1^i,\dots , f_{k_i-1}^i \in H^1(\Omega_i)$,
where $1\leq i\leq j$, associated with the first $k_i$ eigenvalues in $\Omega_i$, denoted by $\mu_0^i \leq \mu_1^i \leq \dots \leq \mu_{k_i -1}^i$. The main property that we will use in the sequel  is that, whenever a function $\varphi \in H^1(\Omega)$ satisfies  $\varphi \in Vect\{f_0^i,f_1^i,\dots f_{k_i-1}^i \}^{\bot}$ (the orthogonality here is intended in $L^2$) then 
$$
\int_{\Omega_i}|\nabla \varphi|^2 \;\mathrm{d}x \geq \mu_{k_i} \int_{\Omega_i} \varphi^2 \;\mathrm{d}x.
$$
This follows from the standard min-max principle  which says that  (see for instance Theorem 3.I.9. page 69 of \cite{Polterovitch}),
$$\mu_{k_i}(\Omega_i)=\min_{u \in H^1(\Omega_i) \cap Vect\{f_0^i,f_1^i,\dots f_{k_i-1}^i \}^{\bot} } \frac{\int_{\Omega_i} |\nabla u|^2 \;\mathrm{d}x}{\int_{\Omega_i} u^2 \;\mathrm{d}x}.$$

Now for each $0\leq \ell \leq k$ we consider a normalized eigenfunction $\varphi_\ell$ associated to the eigenvalue $\mu_\ell$ in the big domain $\Omega$. The space $Vect\{f_\ell^i\}_{i=1,\dots, j}^{\ell=0,\dots, k_i}$ is of dimension $k$ in $L^2(\Omega)$, and on the other hand the space $Vect\{\varphi_0, \dots, \varphi_k\}$ is of dimension $k+1$. Therefore, there exists a function $\varphi\in Vect(\varphi_\ell)$ that lies in the orthogonal of $Vect(f_\ell^i)$. In other words we can find some coefficients $\alpha_\ell$ such that the function 
$$\varphi = \alpha_0 \varphi_0+\cdots +\alpha_k\varphi_k$$
verifies 
$$ \int_{\Omega_i} \varphi f_\ell^i \;\mathrm{d}x=0, \quad \text{ for all }1\leq i\leq j, \text{ and }\quad 1\leq \ell \leq k_i-1.$$
Up to dividing $\varphi$ by its $L^2$ norm we can also assume that $\int_{\Omega} \varphi^2 \;\mathrm{d}x=1$ or put differently, $\sum_{j=0}^k \alpha_j^2 =1$.  We deduce that 
\begin{eqnarray}
\int_{\Omega} |\nabla \varphi|^2 \;\mathrm{d}x &=&\sum_{i=1}^j \int_{\Omega_i} |\nabla \varphi|^2 \;\mathrm{d}x  \notag \\
&\geq& \sum_{i=1}^j \mu_{k_i}(\Omega_i)\int_{\Omega_i} |\varphi|^2 \;\mathrm{d}x\geq \min_i(\mu_{k_i}(\Omega_i)). \label{equality11}
\end{eqnarray}
 On the other hand by orthogonality of $\{\varphi_j\}$ in $H^1(\Omega)$ we have
 \begin{eqnarray}
\int_{\Omega} |\nabla \varphi|^2 \;\mathrm{d}x &=& \sum_{j=0}^k \alpha_j^2\int_{\Omega}  |\nabla \varphi_j|^2 \;\mathrm{d}x \notag \\
&=&  \sum_{j=0}^k \alpha_j^2 \mu_j(\Omega) \notag \\
&\leq & \mu_k(\Omega) \label{equality14}
\end{eqnarray}
which proves that 
\begin{eqnarray}
\mu_k(\Omega)\geq \min_i(\mu_{k_i}(\Omega_i)). \label{equality12}
\end{eqnarray}
Now if equality occurs, then we have equality in \eqref{equality11}  which means that all the values of $\mu_{k_i}(\Omega_i)$ must be equal. Then there is equality also in \eqref{equality12} which means that $\mu_{k_i}(\Omega_i)=\mu_k(\Omega)$, for all $i$. Then there is also equality in \eqref{equality14} which means that $\varphi$ is an eigenfunction for $\mu_k(\Omega)$. Moreover the equality  in \eqref{equality11} says that the value of the Rayleigh quotient of $\varphi$ on $\Omega_i$ is equal to $\mu_{k_i}(\Omega_i)$ thus according to \cite[Theorem 3.I.9, page 69]{Polterovitch} we infer that the restriction of $\varphi$ on $\Omega_i$ must be a Neumann eigenfunction associated to $\mu_{k_i}(\Omega_i)$.
\end{proof}

\section{Squeezing or stretching lemma}

Let us now give a classical result that can be found for example in \cite[Proof of Proposition 8.1]{LS} or \cite[Lemma 6.20]{Lau-Siu}.
Let $t\in (0,1)$ and let  $\psi_t:\R^2\to \R^2$ be  the squeezing mapping defined by $\psi_t(x,y)=(x,ty)$ (we can do exactly the same proof with
$t\geq 1$ and in that case where the eigenvalue decreases this would be a stretching lemma).  For a domain $\Omega$ we denote by $\Omega_t:=\psi_t(\Omega)$ the squeezed domain in the ``vertical direction''.
\begin{lemma} \label{strechL}
For all $k\geq 1$ we have 
$$\mu_k(\Omega)\leq \mu_k(\Omega_t).$$
\end{lemma}

\begin{proof}

We start by noticing that $\psi_t$ is a diffeomorphism from $\R^2$ to $\R^2$ so that all subspaces $H$ of dimension $k$ in $H^1(\Omega)$ are of the form $\psi_t^{-1}(L)$ for a subspace $L$ of dimension $k$ in  $H^1(\Omega_t)$.

Thus let us pick any subspace $L$ in $H^1(\Omega_t)$ and let $w \in L$. Then the function $u=w \circ \psi_t$ belongs to $\psi_t^{-1}(L)=:H$ and by the change of variables formula we have 
 $$ \frac{\int_{\Omega_t} |\nabla w|^2  \;\mathrm{d}x}{\int_{\Omega_t} w^2 \;\mathrm{d}x} = \frac{\int_{\Omega} |\nabla w|^2\circ \psi_t |J\psi_t| \;\mathrm{d}x}{\int_{\Omega}w^2\circ \psi_t  |J\psi_t|\;\mathrm{d}x}.$$
 Now we wish to estimate the above quotient. First of all we remark that  the Jacobian of $\psi_t$ is simply $t$, in other words 
 $$|J\psi_t| =t  \quad \quad |J \psi_t^{-1}| =\frac{1}{t}$$
 so we deduce that 
  $$ \frac{\int_{\Omega_t} |\nabla w|^2  \;\mathrm{d}x}{\int_{\Omega_t} w^2 \;\mathrm{d}x} = \frac{\int_{\Omega} |\nabla w|^2\circ \psi_t \;\mathrm{d}x}{\int_{\Omega}w^2\circ \psi_t \;\mathrm{d}x}= \frac{\int_{\Omega} |\nabla w|^2\circ \psi_t \;\mathrm{d}x}{\int_{\Omega}(w\circ \psi_t)^2 \;\mathrm{d}x}.$$

Now we want to compare $|\nabla w|^2\circ \psi_t$ with $|\nabla (w\circ \psi_t)|^2$ which of course are not the same. It is immediate to check that the function $u:=w \circ \psi_t$ defined in $\Omega$ satisfies
$$
|\nabla u|^2 = |\nabla w|^2 \circ \psi_t + (t^2-1) |\partial_y w|^2 \circ \psi_t.
$$
%write $u_t= w(x,ty)$ and then
%$$\partial_x u_t(x,ty)= \partial_x w(x,ty)$$
%$$\partial_y u_t(x,ty)= t\partial_y w(x,ty)$$ 
%thus 
%$$|\nabla w|^2 \circ \psi_t = |\nabla (w\circ \psi_t)|^2 +(t^2-1) |\partial_y w\circ \psi_t |^2.$$

Therefore we get 
\begin{eqnarray}
 \int_{\Omega} |\nabla u|^2 \; \mathrm{d}x &=& \int_{\Omega} |\nabla w | ^2 \circ \psi_t  \;\mathrm{d}x  + (t^2-1)  \int_{\Omega}  |\partial_y w|^2 \circ \psi_t   \;\mathrm{d}x \notag \\
 &=&  \frac{1}{t}\int_{\Omega_t}  |\nabla w|^2 \;\mathrm{d}x +(t^2-1)  \int_{\Omega}  |\partial_y w|^2\circ \psi_t  \;\mathrm{d}x. \notag
 \end{eqnarray}

 Returning back to the Rayleigh quotient, we obtain that for any $w \in L$,
  $$ \frac{\int_{\Omega} |\nabla (w\circ \psi_t)|^2  \;\mathrm{d}x}{\int_{\Omega}(w\circ \psi_t)^2  \;\mathrm{d}x} \leq   \frac{\int_{\Omega_t} |\nabla w|^2  \;\mathrm{d}x}{\int_{\Omega_t}w^2 \;\mathrm{d}x} + \frac{(t^2-1)  \int_{\Omega}  |\partial_y w\circ \psi_t |^2  \;\mathrm{d}x} {\frac{1}{t}\int_{\Omega_t}w^2 \;\mathrm{d}x} \leq \frac{\int_{\Omega_t} |\nabla w|^2  \;\mathrm{d}x}{\int_{\Omega_t}w^2 \;\mathrm{d}x},
$$
where we have used, in the last inequality, that $t\in (0,1)$.
By taking now the maximum in the $w$ variable we arrive at
$$\max_{w\in L}\frac{\int_{\Omega_t} |\nabla w|^2  \;\mathrm{d}x}{\int_{\Omega_t}w^2 \;\mathrm{d}x}  \geq    \max_{v \in \psi_t^{-1}(L)} \frac{\int_{\Omega} |\nabla v|^2 \;\mathrm{d}x}{\int_{\Omega}v^2 \;\mathrm{d}x }.
$$
Passing to the min in $L$ yields
$$\mu_k(\Omega_t)\geq \mu_k(\Omega),$$
as desired.
\end{proof}
 
 \section{An explicit Poincar\'e inequality for planar domains with cusps}
 
Let us give a Poincar\'e inequality for our non-convex planar domain $\Omega_2$ appearing in the proof of Theorem \ref{doubmax}. For this section, for brevity, we will simply write $\Omega$ and $D$ instead of $\Omega_2$ and $D^*$.

For the construction, we start with choosing a point $P$ on the boundary $\partial D$ of  the convex domain $D$.
Then we choose a supporting line $L$ of $\partial D$ at $P$ and choose the $y$-axis as the exterior normal line of $\partial D$ at $P$ which is orthogonal to $L$.
 Next we choose a point $Q$ on the $y$-axis outside $D$ and finally find the two tangent lines (or supporting lines) of $D$ through $Q$.
This construction allows us to locally express a portion of $\partial D$ as the graph of a concave
function $f:[-a,b]\to\mathbb{R}$, ($a,b >0$)
having its maximum value $f(0)$ where $f$ increases in $[-a, 0]$ and decreases in $[0,b]$.
\begin{figure}[h]
\centering
\includegraphics[scale=0.2]{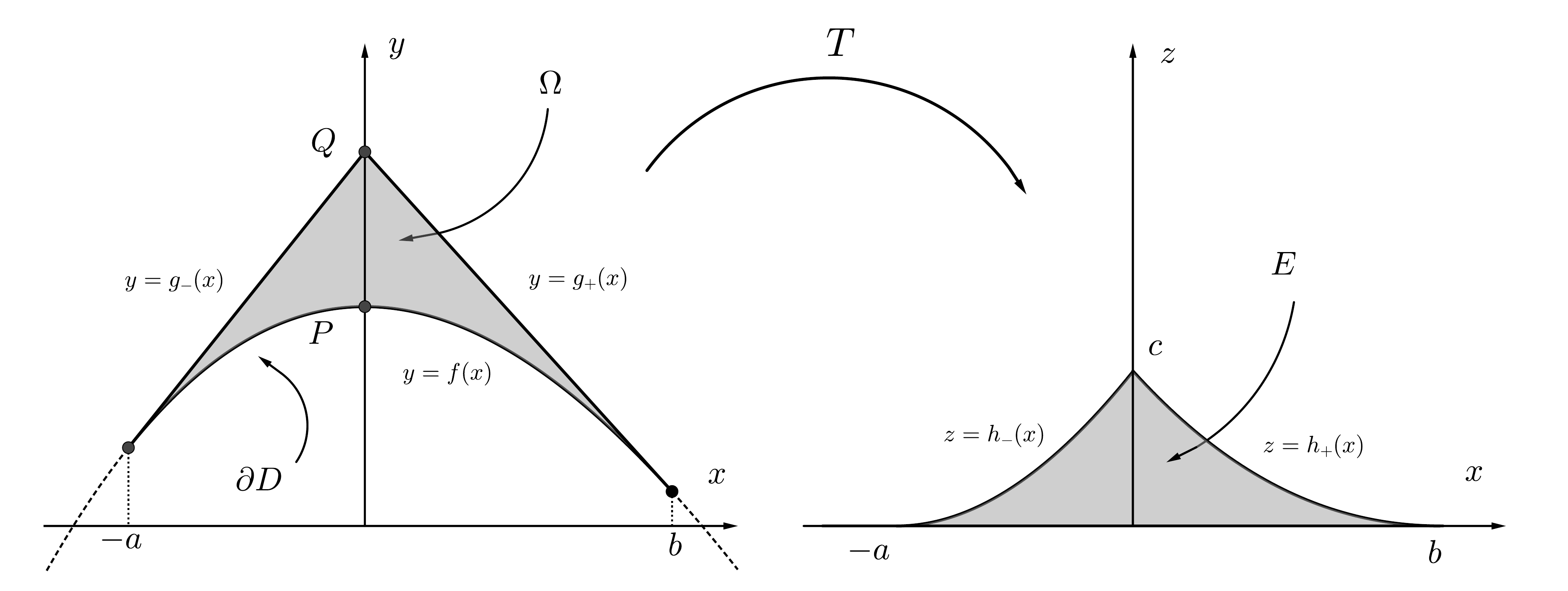}
\caption{Construction of the concave function $f$.}\label{fig-Poincare}
\end{figure}
By construction the two tangent lines of the curve $y=f(x)$ at points $(-a,f(-a)), (b,f(b))$ intersect at the point $Q$ on $y-$axis. Let these tangent lines be given  by $y=g_-(x), y=g_+(x),$ respectively. 
Define a piecewise linear function $g : [-a,b] \to \mathbb R$ by 
\begin{equation}
\label{piecewise linear functions}
g(x) =
\begin{cases} 
g_-(x)\ \text{ if } -a\le x \le 0,\\
g_+(x)\ \text{ if } 0< x \le b.
\end{cases}
\end{equation}
Assume that $g(x) > f(x)$ if $-a < x < b$. Set $h(x) = g(x)-f(x)$ for $x\in[-a,b]$. Notice that $h$  increases in $[-a, 0]$ and decreases in $[0,b]$. Set $c = h(0) (= \max h > 0)$. 
Let $\Omega$ be the non-convex planar domain given by
$$
\Omega=\{ (x,y) : -a< x < b, f(x) < y < g(x) \}.
$$

Introduce the plane transformation 
\begin{equation}\label{mapT}
T: (x, y) \mapsto (x, z) = (x, y-f(x)).
\end{equation}
Then $\Omega$ corresponds to the planar domain $E$ given by
$$
E:=T(\Omega)=\{ (x,z) : -a < x <b, 0 < z < h(x) \}.
$$
The set  $E$ is non-convex, since $h$ is convex in each of the intervals $[-a, 0]$ and $[0,b]$, and it has its maximum at $0$. Notice that  $|E|=|\Omega|$  where $|\cdot|$ denotes the area.

For convenience, corresponding to \eqref{piecewise linear functions}, we write
$$
h(x) =
\begin{cases} 
h_-(x)\ \text{ if } -a\le x \le 0,\\
h_+(x)\ \text{ if } 0< x \le b.
\end{cases}
$$

Let us first prove a Poincar\'e inequality for $E$. Then, by virtue of $T$, we may get a Poincar\'e inequality for $\Omega$.

Let $(x_1, z_1), (x_2,z_2) \in E$. Distinguish two cases: (i) $ z_1 < z_2$, (ii) $z_1 \ge z_2$.
In case (i) $(x_2,z_1)\in E$ and in case (ii) $(x_1,z_2)\in E$. Let $u \in C^1(E)\cap H^1(E)$.

In case (i), we have
\begin{eqnarray}
u(x_1,z_1) - u(x_2,z_2) &=&u(x_1,z_1) - u(x_2,z_1) + u(x_2,z_1)- u(x_2,z_2)\nonumber\\
&=&\int_{x_2}^{x_1}\partial_xu(x,z_1)\mathrm{d}x + \int_{z_2}^{z_1}\partial_zu(x_2,z)\mathrm{d}z\ (= L_1).\label{case (i)}
\end{eqnarray}
In case (ii), we have
\begin{eqnarray}
u(x_1,z_1) - u(x_2,z_2) &=&u(x_1,z_1) - u(x_1,z_2) + u(x_1,z_2)- u(x_2,z_2)\nonumber\\
&=&\int_{z_2}^{z_1}\partial_zu(x_1,z)\mathrm{d}z+ \int_{x_2}^{x_1}\partial_xu(x,z_2)\mathrm{d}x\ (= L_2).\label{case (ii)}
\end{eqnarray}
For each $(x_1, z_1) \in E$, we integrate \eqref{case (i)} in $(x_2, z_2) \in E \cap \{  z_1 < z_2\}$ and \eqref{case (ii)}  in $(x_2, z_2) \in E \cap \{  z_1\ge z_2\}$ and then sum the resulting equations to have 
\begin{eqnarray}
|E|( u(x_1,z_1) - u_{E})=\int_{E \cap \{  z_1 < z_2\}}L_1\mathrm{d}x_2\mathrm{d}z_2 + \int_{E \cap \{  z_1 \ge z_2\}}L_2 \mathrm{d}x_2 \mathrm{d}z_2,\nonumber
\end{eqnarray}
where we set $u_{E}= \frac 1{|E|}\int_{E}u(x,z)\mathrm{d}x\mathrm{d}z$.
Then, the Schwarz  inequality gives
\begin{equation}
\label{due to Schwarz}
|u(x_1,z_1) - u_{E}|^2\le \frac 2{|E|}\left\{ \int_{E \cap \{  z_1 < z_2\}}(L_1)^2 \mathrm{d}x_2 \mathrm{d}z_2 + \int_{E \cap \{  z_1 \ge z_2\}}(L_2)^2 \mathrm{d}x_2 \mathrm{d}z_2\right\}.
\end{equation}
Here, the Schwarz inequality applied to \eqref{case (i)} and \eqref{case (ii)} gives also the following:
\begin{eqnarray}
\frac12(L_1)^2\!\!&\le&\!\! |x_1\!-\!x_2|\left|\int_{x_2}^{x_1}|\partial_xu(x,z_1)|^2 \mathrm{d}x\right|+|z_1\!-\!z_2|\left|\int_{z_2}^{z_1}|\partial_zu(x_2,z)|^2 \mathrm{d}z\right|\nonumber\\
\!\!&\le&\!\!(a\!+\!b)\!\!\int_{h_-^{-1}(z_1)}^{h_+^{-1}(z_1)}\!\!|\partial_xu(x,z_1)|^2 \mathrm{d}x+c\int_{0}^{h(x_2)}\!\!\!\!|\partial_zu(x_2,z)|^2 \mathrm{d}z\ (= M_1),\label{inequality for L1}\\
\frac12(L_2)^2\!\!&\le&\!\! |z_1\!-\!z_2|\left|\int_{z_2}^{z_1}|\partial_zu(x_1,z)|^2 \mathrm{d}z\right|+|x_1\!-\!x_2|\left|\int_{x_2}^{x_1}|\partial_xu(x,z_2)|^2 \mathrm{d}x\right|\nonumber\\
\!\!&\le&\!\!c\int_{0}^{h(x_1)}\!\!\!\!|\partial_zu(x_1,z)|^2 \mathrm{d}z+(a\!+\!b)\!\!\int_{h_-^{-1}(z_2)}^{h_+^{-1}(z_2)}\!\!|\partial_xu(x,z_2)|^2 \mathrm{d}x\ (= M_2).\label{inequality for L2}
\end{eqnarray}
Hence integrating \eqref{due to Schwarz} in $(x_1, z_1) \in E$ and combining three inequalities \eqref{due to Schwarz}, \eqref{inequality for L1}, \eqref{inequality for L2} yield that
\begin{eqnarray}
\int_E|u(x_1,z_1) - u_{E}|^2\mathrm{d}x_1\mathrm{d}z_1&\le&\frac 4{|E|}\int_E\mathrm{d}x_1\mathrm{d}z_1\int_E \mathrm{d}x_2 \mathrm{d}z_2\left\{M_1+M_2\right\}\nonumber\\
&\le& 8\left\{(a\!+\!b)^2\int_E(\partial_xu)^2\mathrm{d}x\mathrm{d}z+c^2\int_E(\partial_zu)^2\mathrm{d}x\mathrm{d}z\right\}\nonumber\\
&\le& 8\max\{(a\!+\!b)^2,c^2\}\int_E|\nabla u|^2\mathrm{d}x\mathrm{d}z.\nonumber %\label{conclusion inequality}
\end{eqnarray}
Thus we obtain
\begin{proposition}[Poincar\'e inequality for $E$] For every $u \in H^1(E)$,
$$
\int_E|u - u_{E}|^2\mathrm{d}x\mathrm{d}z\le 8\max\{(a\!+\!b)^2,c^2\}\int_E|\nabla u|^2\mathrm{d}x\mathrm{d}z.
$$
\end{proposition}
Then, by virtue of $T$ defined in \eqref{mapT}, we may get the following Poincar\'e inequality for $\Omega$.
\begin{proposition}[Poincar\'e inequality for $\Omega$] For every $v \in H^1(\Omega)$,
$$
\int_\Omega|v - v_{\Omega}|^2\mathrm{d}x\mathrm{d}y\le 16(L^2+1)\max\{(a\!+\!b)^2,c^2\}\int_\Omega|\nabla v|^2\mathrm{d}x\mathrm{d}y,
$$
where $L$ is the Lipschitz constant of the function $f$.
\end{proposition}

\begin{proof} 
Let $v \in H^1(\Omega)$. Set $u(x,z) = v(x, f(x)+z)$. Then $u \in H^1(E)$, since $f$ is Lipschitz continuous.
Since the Jacobian of the transformation $T$ equals $1$, we have
$$
|E| = |\Omega|, \quad u_E=v_\Omega,  \quad \int_E|u - u_{E}|^2\mathrm{d}x\mathrm{d}z=\int_\Omega|v - v_{\Omega}|^2\mathrm{d}x\mathrm{d}y.
$$
Observe that
$$
|\nabla u|^2=u_x^2+u_z^2=(v_x+v_yf^\prime(x))^2+v_y^2\le 2v_x^2+2v_y^2(f^\prime(x))^2+v_y^2\le2(L^2+1)|\nabla v|^2.
$$
Hence 
$$
\int_E|\nabla u|^2\mathrm{d}x\mathrm{d}z\le 2(L^2+1)\int_\Omega|\nabla v|^2\mathrm{d}x\mathrm{d}y.
$$
Therefore a Poincar\'e inequality for $E$ yields a Poincar\'e inequality for $\Omega$.
\end{proof}

We conclude that 
\begin{corollary} 
$$
\mu_1(\Omega) \ge \frac 1{ 16(L^2+1)\max\{(a\!+\!b)^2,c^2\}}.
$$
\end{corollary}

\bibliography{biblio}
\bibliographystyle{plain}
\end{document}